\documentclass[reqno,11pt]{amsart}
\usepackage{amsmath,amssymb,amsthm,graphicx,a4wide,enumerate,url}
\usepackage[small,bf]{caption} 
\usepackage{amsmath}
\usepackage{color}
\usepackage{lipsum}
\usepackage{amsfonts}
\usepackage{graphicx}
\usepackage{epstopdf}
\usepackage{algorithmic}
\usepackage[outercaption]{sidecap}
\usepackage{amsopn}
\usepackage{enumerate}
\usepackage{array}
\usepackage{makecell}
\usepackage{adjustbox}
\usepackage{calc}
\usepackage{accents}
\usepackage{hyperref}
\usepackage{cleveref}

\renewcommand{\O}{{\mathcal O}}
\DeclareMathOperator{\diag}{diag}

\newtheorem{theorem}{Theorem}[section]
\newtheorem{remark}[theorem]{Remark}
\newtheorem{example}[theorem]{Example}
\newtheorem{lemma}[theorem]{Lemma}
\newtheorem{corollary}[theorem]{Corollary}
\newtheorem{definition}[theorem]{Definition}

\begin{document}
\parskip.9ex

\title[Algebraic Structure of Weak Stage Order]
{Algebraic Structure of the Weak Stage Order Conditions for Runge-Kutta Methods}

\author[A. Biswas]{Abhijit Biswas}
\address[Abhijit Biswas]
{Computer, Electrical, and Mathematical Sciences \& Engineering Division \\ 
King Abdullah University of Science and Technology \\ Thuwal 23955 \\ Saudi Arabia} 
\email{abhijit.biswas@kaust.edu.sa}
\urladdr{https://math.temple.edu/\~{}tug14809}

\author[D. Ketcheson]{David Ketcheson}
\address[David Ketcheson]
{Computer, Electrical, and Mathematical Sciences \& Engineering Division \\ 
King Abdullah University of Science and Technology \\ Thuwal 23955 \\ Saudi Arabia} 
\email{david.ketcheson@kaust.edu.sa}
\urladdr{https://www.davidketcheson.info}

\author[B. Seibold]{Benjamin Seibold}
\address[Benjamin Seibold]
{Department of Mathematics \\ Temple University \\
1805 North Broad Street \\ Philadelphia, PA 19122}
\email{seibold@temple.edu}
\urladdr{http://www.math.temple.edu/\~{}seibold}

\author[D. Shirokoff]{David Shirokoff}
\address[David Shirokoff]
{Corresponding author, Department of Mathematical Sciences \\ New Jersey Institute of Technology \\ University Heights \\ Newark, NJ 07102}
\email{david.g.shirokoff@njit.edu}
\urladdr{https://web.njit.edu/\~{}shirokof}


\begin{abstract}
Runge-Kutta (RK) methods may exhibit order reduction when applied to stiff problems. For linear problems with time-independent operators, order reduction can be avoided if the method satisfies certain weak stage order (WSO) conditions, which are less restrictive than traditional stage order conditions. This paper outlines the first algebraic theory of WSO, and establishes general order barriers that relate the WSO of a RK scheme to its order and number of stages for both fully-implicit and DIRK schemes. It is shown in several scenarios that the constructed bounds are sharp. The theory characterizes WSO in terms of orthogonal invariant subspaces and associated minimal polynomials. The resulting necessary conditions on the structure of RK methods with WSO are then shown to be of practical use for the construction of such schemes.
\end{abstract}

\maketitle
\renewcommand{\O}{{\mathcal O}}
\newcommand{\dbtilde}[1]{\accentset{\approx}{#1}}
\newcommand{\vardbtilde}[1]{\tilde{\raisebox{0pt}[0.85\height]{$\tilde{#1}$}}}

\newcommand{\myremarkend}{$\spadesuit$} 
\newcommand{\lbmap}{\left[} 
\newcommand{\rbmap}{\right]} 
\newcommand{\WSOset}{\mathbb{W}_q}
\newcommand{\OCset}{\mathbb{V}_{p}}
\newcommand{\numc}{n_c}
\newcommand{\minPolc}{P_{\rm C}}
\newcommand{\qso}{\tilde{q}}
\newcommand{\RzCoeff}{\beta}
\newcommand{\Sredr}{\tilde{r}}
\newcommand{\BigO}{{\mathcal{O}}}
\newcommand{\Dt}{\Delta t}
\newcommand{\vecpower}[2]{\vec{#1}^{\,#2}}
\newcommand{\kgen}{m}
\newcommand{\basismat}{L}
\newcommand{\Umat}{U}
\newcommand{\Vmat}{V}
\newcommand{\Wmat}{W}
\newcommand{\bsmall}{\hat{\vec{b}}}
\newcommand{\Ksmall}{\hatK}
\newcommand{\kibitz}[2]{\textcolor{#1}{#2}}
\newcommand{\ben}[1] {\kibitz{magenta}{[BS says: #1]}}
\newcommand{\abhi}[1] {\kibitz{blue}{[AB says: #1]}}
\newcommand{\dave}[1] {\kibitz{red}{[DS says: #1]}}
\newcommand{\david}[1] {\kibitz{purple}{[DK says: #1]}}
\newcommand{\red}[1]{\kibitz{red}{#1}}

\section{Introduction}
\label{sec:introduction}

Runge-Kutta (RK) methods may exhibit a convergence rate lower than the (classical) order of the scheme. This \emph{order reduction} phenomenon often occurs when integrating stiff ODEs, stiff PDEs, or initial-boundary value problems (IBVPs) with time-dependent boundary conditions \cite{burrage1990order, AlonsoPalencia2003, crouzeix1980approximation, Sanz-Serna1986, VerwerSanzSerna1984, OstermannRoche1992, LubichOstermann1995quasilinear, OstermannThalhammer2002}. 
In the worst case, the observed convergence rate is governed by the \emph{stage order} of the method \cite{prothero1974stability, wanner1996solving}. Unfortunately, RK methods with high stage order must be fully implicit; diagonally implicit RK (DIRK) schemes cannot have stage order above 2.

By considering a restricted class of problems, one can obtain conditions weaker than stage order that are sufficient to ensure high-order convergence.  For instance, in some cases where order reduction stems from time-dependent PDE boundary conditions, it can be avoided by various problem- and method-dependent modifications \cite{Carpenter1995, abarbanel1996removal, alonso2002runge, CalvoPalencia2002, alonso2004avoiding, Calvo2002}. A set of conditions ensuring high-order convergence for linear problems has been developed \cite{OstermannRoche1992, Hundsdorfer1986, BurrageHundsdorferVerwer1986, scholz1989order} and recently referred to by the term \emph{weak stage order} \cite{rosales2017spatial, KetchesonSeiboldShirokoffZhou2020}.  The conditions for weak stage order are, as the name suggests, weaker than those for stage order; notably, DIRK methods can have high weak stage order. The conditions are closely related to the classical order conditions for RK methods. In this work we establish the algebraic structure of weak stage order and use it to prove rigorous relationships between the achievable order, the weak stage order, and the number of stages of a method.

\subsection{The Weak Stage Order Conditions}
\label{subsec:intro_WSO}
We consider the initial value problem
\begin{equation}\label{eq:IVP}
u^{\prime}(t)  = f(t,u(t)), \  u(0) = u_{0}; \ u\in \mathbb{R}^{m}, \ f : \mathbb{R}\times\mathbb{R}^m \to \mathbb{R}^m ,
\end{equation}
and its numerical approximation $u_n \approx u(t_n)$, where $t_n = n\Delta t$, via an $s$-stage Runge-Kutta (RK) scheme, where the one-step update rule
\begin{subequations}\label{eq:RK_methods}
\begin{align}\label{eq:RK_methods_1}
    g_i & = u_n+\Delta t \sum_{j = 1}^{s} a_{ij} f(t_n+c_j \Delta t, g_j)\;, \quad i = 1,2, \ldots,s\; \\ \label{eq:RK_methods_2}
    u_{n+1} & = u_{n}+\Delta t \sum_{j = 1}^{s} b_j f(t_n+c_j \Delta t, g_j)\;,
\end{align}
\end{subequations}
is concisely represented via the Butcher tableau
\begin{equation*}
\renewcommand\arraystretch{1.2}
\begin{array}
{c|c}
\vec{c} & A\\
\hline
& \vecpower{b}{T}
\end{array}\;, \; \textrm{where} \; 
A = (a_{ij})_{i,j=1}^s, \quad \vec{b} = (b_1, \ldots, b_s)^T, \quad \vec{c} = (c_1, \ldots, c_s)^T.
\end{equation*}
Method \eqref{eq:RK_methods} is said to be of classical order $p$ if the one-step error \cite{HairerNorsettWanner1993} is of $\mathcal{O}(\Delta t^{p+1})$ as $\Delta t \rightarrow 0$. This requires that the coefficients of the method satisfy certain algebraic relations known as order conditions, which are derived under the assumption that $f$ has Lipschitz constant $\lambda$, and that $\Delta t \lambda \ll 1$.

If one does not assume that $\Delta t \lambda$ is small (i.e., in the study of stiff problems),  a careful expansion of the numerical error leads to additional algebraic conditions on the scheme's coefficients. A necessary condition for convergence of order $\xi$ for general stiff ODEs (cf. \cite[Theorem~15.3]{wanner1996solving}) is that 
\begin{equation}\label{Eq:SOV_MatrixForm}
    \vec{\tau}^{(k)} := A \vecpower{c}{k-1}-\tfrac{1}{k}\vecpower{c}{k} = \left(A C^{k-1} - \tfrac{1}{k} C^{k} \right) \vec{e} = 0, 
    \quad 
    \textrm{for } 1 \le k \leq \xi\;.
\end{equation}
Here $\vecpower{c}{k} := \left(c_{1}^{k},c_{2}^{k}, \ldots, c_{s}^{k} \right)^T$ denotes component-wise exponentiation, the vector $\vec{e} = (1, \ldots, 1)^T \in \mathbb{R}^{s}$, and the diagonal matrix $C = \diag(c_1, c_2, \ldots, c_s)$ so that $\vec{c} = C\vec{e}$. Throughout this paper we assume that $\vec{c} = A\vec{e}$, which implies that $\tau^{(1)} = 0$.
We refer to the vectors $\vec{\tau}^{(k)}$
as \emph{stage order residuals}.

If one assumes that $f$ is a linear (possibly time-dependent) operator, 
then in the convergence analysis for stiff problems one can replace the stage order with the largest integer $\xi$ such that
\begin{equation}\label{eq:OR_WSO}
    S(\xi):\quad\quad \vecpower{b}{T} A^j \vec{\tau}^{(k)} = 0, \quad \textrm{for all} \; 0 \leq j \leq s-1, \; 1 \leq k \leq \xi\;.
\end{equation}
\begin{definition}(Weak stage order, WSO)\label{dfn:wso1}
	The weak stage order $q$ of an $s$-stage RK scheme $(A,\vec{b})$ is the largest integer such that condition $S(q)$ in \eqref{eq:OR_WSO} holds. If $S(\xi)$ holds for every $\xi \geq 1$, then $q = \infty$.
\end{definition}

Conditions \eqref{eq:OR_WSO} were first introduced in the context of ROW methods, and referred to as \emph{parabolic order conditions} \cite{OstermannRoche1992, OstermannRoche1993}; see \S\ref{eq:pro_robin_stiff_problem}.
In \S\ref{sec:WSO-review}, we review the derivation of condition \eqref{eq:OR_WSO} as well as what is known about the class of problems for which this condition ensures convergence of order $q$. We also provide a review of the literature relating to WSO and similar conditions.

\subsection{Contribution of this Work}
\label{subsec:Intro_ContributionWork}
While much is known about RK order conditions, less is known about the solvability of the WSO conditions, and more generally, the simultaneous solution of both sets of order conditions. Since high stage order methods must be fully implicit, the solvability of the WSO and order conditions for DIRK schemes is of particular interest.

Previous attempts at constructing DIRK schemes with WSO \cite{rosales2017spatial, KetchesonSeiboldShirokoffZhou2020} relied upon the simplifying assumption that the vectors $\tau^{(k)}$ in \eqref{eq:OR_WSO} be eigenvectors of $A$. For instance, DIRK schemes up to order $4$ and WSO $3$ have been provided in \cite{KetchesonSeiboldShirokoffZhou2020}. This simplifying ``eigenvector criterion'' limits the construction of WSO schemes for DIRKs (with invertible $A$) to $q \leq 3$, as shown below. A key motivating factor for the present work is the need to construct DIRK schemes with high WSO $q > 3$; in order to enable this, herein the vectors $\vec{\tau}^{(k)}$ are not assumed to be eigenvectors of $A$. 
Although we are particularly interested in DIRK schemes, some of the main results herein apply to general RK schemes.

This paper provides the first known bounds relating a RK scheme's WSO $q$ to the order $p$ and number of stages $s$. After reviewing the background on WSO (\S\ref{sec:WSO-review}),  we recast WSO in terms of orthogonal invariant subspaces related to the matrix $A$ (\S\ref{sec:WSO_Recast_As_OrthSpaces}). From there we establish general bounds on $s$, $q$, $p$ for both fully implicit RK schemes and refinements for DIRKs.  
In \S\ref{Sec:MinPoly} we introduce minimal polynomials corresponding to the orthogonal invariant subspaces. The minimal polynomials are then combined with a family of polynomials (orthogonal with respect to a linear functional) to yield new formulas (\S\ref{Subsec:stabilityfunction}) for the stability function, relevant for schemes with high WSO. The study also reveals necessary conditions for WSO $q > 3$ (\S\ref{sec:DIRKscheme_Results}), which provides the theoretical foundation for devising new schemes in the companion paper \cite{BiswasKetchesonSeiboldShirokoff2023}. We conclude with examples (\S\ref{sec:Examples}) demonstrating the sharpness of the bounds.

\section{Relevance of Weak Stage Order}
\label{sec:WSO-review}
This section provides background information relevant for this work: mathematical notation and fundamentals (\S\ref{subsec:notation}), the manifestations of order reduction and error formulas (\S\ref{subsec:WSO}), and a literature review of prior work related to WSO (\S\ref{subsec:Intro_Schemes_AvoidOR}). While the expert reader familiar with the literature may skip this section, it provides a motivation for the relevance of this paper.

\subsection{Fundamentals and Notation}
\label{subsec:notation}
The following definitions and notation about RK schemes are used in this manuscript.

Given a RK method, we let $\numc$ denote the number of distinct abscissas $c_j$, e.g., if $c_1 = \dots = c_s$ then $\numc = 1$; in turn, if all $c_j$ are distinct (i.e., the scheme is non-confluent) then $\numc = s$.

Schemes for which $A$ is lower-triangular ($a_{ij} = 0$ for $j > i$) are called \emph{diagonally implicit Runge-Kutta} (DIRK) methods. If furthermore all diagonal entries are equal, i.e., $a_{ii} = \gamma$ for $i = 1, \dots, s$, then the scheme is called \emph{singly diagonally implicit} (SDIRK). We refer to DIRK methods with $a_{11} =0$ as EDIRKs.  Implicit schemes that are not diagonally implicit are referred to as \emph{fully implicit}.

The following definition generalizes the concept of an EDIRK, and will be useful in providing some results below (such as \Cref{Thm:MainResultDIRK})  that are less strict for certain schemes with an explicit stage.
\begin{definition} (GEDIRK)
We call a DIRK scheme a generalized EDIRK (or GEDIRK), if $\vec{c}$ contains at least one zero and $a_{\ell \ell} = 0$, where $\ell = \min\{ j \: | \; c_j = 0\}$ is the index of the first zero in $\vec{c}$.
\end{definition}

When a RK scheme is applied to the scalar ODE $u'(t) = \lambda u$, it results in the iteration $u_{n+1} = R(z) u_n$, where $z := \lambda \Delta t$, and
\begin{equation}\label{eq:stabilityfunction1}
    R(z) := 1 + z\vecpower{b}{T} (I - z A)^{-1} \vec{e} = 
    \frac{\det(I-zA+z\vec{e}\vecpower{b}{T})}{\det(I-zA)}
\end{equation}
is the \emph{stability function} of the scheme.

A RK scheme is called \emph{stiffly accurate} \cite{prothero1974stability} if the last row of $A$ equals $\vecpower{b}{T}$, i.e., $a_{sj} = b_j$ for $j=1,\dots,s$. When $A$ is invertible, stiff accuracy implies $\vecpower{b}{T} A^{-1} \vec{e} = 1$, and hence the desirable property $\lim_{z \rightarrow \infty} R(z) = 0$ results.

The \emph{(classical) order $p$} of a RK scheme (see \S\ref{subsec:intro_WSO}) imposes the well-known \emph{order conditions} on the Butcher coefficients \cite{HairerNorsettWanner1993}. 
We will also make use of the conditions
\begin{align}
\label{eq:simplifying_assumptions_B}
B(\xi):& & \vecpower{b}{T} \vecpower{c}{k-1} &=\tfrac{1}{k} & \text{~for~~} &k = 1,2,\ldots,\xi\;, \\
\label{eq:simplifying_assumptions_C}
C(\xi):& & \vec{\tau}^{(k)} &= 0  & \text{~for~~} &k = 1, 2, \ldots, \xi\;,
\end{align}
which determine the accuracy of the quadrature and subquadrature rules (respectively) on which the RK method is based \cite{butcher2008numerical}. 

A RK scheme is said to have \emph{stage order} $\qso := \min\{q_1, q_2\}$ where $q_1$, $q_2$ are the largest integers such that $B(q_1)$ and $C(q_2)$ hold. Stage order $\qso$ guarantees that each RK stage is accurate to order $\qso$, and that the method itself has order at least $\qso$.

\subsection{Order Reduction and the Role of Weak Stage Order}
\label{subsec:WSO}
As the name suggests, the conditions for WSO $q$ are a relaxation of those for stage order $\qso$, i.e., $\qso \leq q$ holds for any method. The relevance of WSO in the context of order reduction can be understood by examining the error for the (stiff) Prothero-Robinson problem \cite{prothero1974stability, wanner1996solving},
\begin{equation}\label{eq:pro_robin_stiff_problem}
    u^{\prime} = \lambda\left(u-\phi(t)\right)+\phi^{\prime}(t), 
    \quad u(0) = \phi(0),\quad \text{with} \ \text{Re}(\lambda) < 0\;,
\end{equation}
with solution $u(t)=\phi(t)$.
For solutions starting from a different initial value,
the difference $u(t) - \phi(t)$ decays exponentially to zero with time.  Problem \eqref{eq:pro_robin_stiff_problem} can be made arbitrarily stiff by choosing $|\lambda|$ large. As in \cite[(15.9) Chap.~IV.15]{wanner1996solving}, one can derive exact formulas for the RK scheme's \emph{global error} (with equal time-steps):
\begin{equation*}
    u_{n+1} - \phi(t_{n+1}) = (R(z))^{n+1} ( u_0 - \phi(0) ) + \sum_{j = 0}^{n} (R(z))^{n-j} \vec{\delta}_{\Delta t}(t_j)\;.
\end{equation*}
Here $z = \lambda \Delta t$ and $R(z)$ are defined as in \eqref{eq:stabilityfunction1}. The \emph{local error} $\vec{\delta}_{\Delta t}(t_n)$ has, for a $p$-th order method, the form
\begin{equation}
    \label{eq:LTE_1}
    \vec{\delta}_{\Delta t}(t_n) = -\sum_{k \geq 1} \frac{( \Delta t)^{k}}{(k-1)!} W_k(z) \phi^{(k)}(t_n)
    +
    \BigO(\Delta t^{p+1})\;,     
\end{equation}
where $\phi^{(k)}(t_n) $ is the $k$-th derivative of $\phi$ at $t_n$ and $W_k(z)$ is the function
\begin{equation*}
    W_k(z) := z \vecpower{b}{T} (I - z A)^{-1} \vec{\tau}^{(k)} \textrm{~~for~~} k \geq 1\;.    
\end{equation*}
In the classical RK convergence theory one examines the non-stiff limit where both the time scale $1/|\lambda|$ and the solution $\phi(t)$ are fully resolved. If one assumes that $\Delta t \to 0$ with $z = \mathcal{O}(\Delta t)$, then one can show that the local error $\vec{\delta}_{\Delta t}(t_n) = \BigO(\Delta t^{p+1})$ by using a Neumann series for $(I-zA)^{-1}$ and applying the order conditions.

In a stiff problem setting, one fully resolves the slow time scale defined by the solution $\phi(t)$ (i.e., $\Delta t |\phi'(t)| \ll 1$) but under-resolves the fast time scale $1/|\lambda|$ (i.e., $|z| = \Delta t |\lambda| > \BigO(1)$). One can then investigate the local error $\vec{\delta}_{\Delta t}(t_n)$ uniformly in $z$. As a special case this in particular covers the \emph{stiff limit} given by $\Delta t \rightarrow 0$ while simultaneously $\lambda \rightarrow -\infty$ such that $z\rightarrow -\infty$. Then the local error $\vec{\delta}_{\Delta t}(t_n)$ is at least order $\qso + 1$ because $\vec{\tau}^{(k)} = 0$ for $k \leq \qso$, but in general it will fail to be of order $p+1$. The algebraic condition of WSO ensures that the functions $W_k(z)$ vanish, yielding a local error $\vec{\delta}_{\Delta t}(t_n) = \BigO(\Delta t^{q+1})$. 
WSO is the least restrictive condition to guarantee $W_k(z) = 0$ for $k \leq q$ \cite[Thm.~2]{KetchesonSeiboldShirokoffZhou2020}.

While the Prothero-Robinson problem \eqref{eq:pro_robin_stiff_problem} provides an intuitive example where the convergence of the local truncation error is facilitated by WSO, both the order reduction phenomenon and its remedy via WSO apply much more generally.

A general class of problems for which WSO improves the convergence order are effectively established by the results in 
\cite{alonso2003optimal}: consider any linear PDE
\begin{equation}
\label{eq:general_PDE}
    u_t = \mathcal{L}u + f(t) \;, \quad \textrm{with boundary data} \quad \mathcal{B}u = g(t)\;,
\end{equation}
and suitable initial data, where $(\mathcal{L},\mathcal{B})$ generate a $C_0$-semigroup. While the theorems in \cite{alonso2003optimal} are stated assuming that the RK scheme used to approximate \eqref{eq:general_PDE} has stage order $q$, the proofs in fact only use the conditions for weak stage order $q$. Hence, WSO is of potential relevance for any PDE problem generated by a $C_0$-semigroup. The results in \cite{alonso2003optimal} build on previous work that focused on analytic semigroups \cite{OstermannRoche1992, OstermannRoche1993},
extending their applicability to a broader class of PDEs---notably wave equations.

For PDE problems, order reduction and WSO may manifest in interesting ways that are not encountered in the ODE problem \eqref{eq:pro_robin_stiff_problem}. For instance, the convergence order is controlled by the WSO of the method, plus a correction that may be fractional depending on the details of the PDE \cite{OstermannRoche1992}.  The fractional convergence can also be understood through the creation of spatial boundary layers in the RK error \cite{rosales2017spatial}. While WSO guarantees the remedy of order reduction only for linear problems, it is also observed numerically that WSO can improve the convergence order for certain nonlinear PDEs where the highest derivatives are linear \cite{KetchesonSeiboldShirokoffZhou2020}.

\subsection{Prior Work on Avoiding Order Reduction}
\label{subsec:Intro_Schemes_AvoidOR}
The prior section highlights that for certain classes of stiff problems the order of the local error may be reduced, in the worst case, to the WSO of the method.

Ostermann \& Roche \cite{OstermannRoche1992} showed that, for linear problems, time-stepping schemes that satisfy
\begin{equation}\label{eq:ConditionOR}
    \tilde{W}_k(z) := \frac{k \; \vecpower{b}{T} (I - z A)^{-1} \vec{\tau}^{(k)}}{R(z) - 1}
    \equiv 0 \quad \textrm{for} \; k = 1, 2, \ldots, q\;,
\end{equation}
along with other technical assumptions, avoid order reduction when applied to linear parabolic PDEs (see \cite{OstermannRoche1992} for RK methods and \cite{OstermannRoche1993} for the related Rosenbrock W- (ROW) methods). The analysis considers PDE problems \eqref{eq:general_PDE} where $\mathcal{L}$ defines an analytic semigroup \cite{OstermannRoche1993} or has a point spectrum in the left half plane with a basis of eigenfunctions \cite{OstermannRoche1992}. The resulting error convergence rate is $\min\{p, q +2 + \bar{\nu}\}$, where $\bar{\nu}$ depends on the PDE, see \cite[Thm.~2]{OstermannRoche1992}.

Hundsdorfer showed that a generalization of \eqref{eq:ConditionOR} implies a convergence rate of $O(\Delta t^{q+1})$ for linear-implicit RK schemes applied to linear ODEs \cite[Thm.~3.3]{Hundsdorfer1986}\footnote{See also \cite[Thm.~3.3]{BurrageHundsdorferVerwer1986} which used $\tilde{W}_k(z)$ to achieve improved convergence estimates.}.
Montijano showed that a necessary condition for convergence at order $q$ on the problem \eqref{eq:pro_robin_stiff_problem} is $b^T A^{-1}c^j = 1$ for $0\le j\le q$ \cite{montijano1983, dekker1984}.

Because \eqref{eq:ConditionOR} contains the variable $z$, it is somewhat impractical as a means to design RK schemes. Scholz \cite{scholz1989order} converted \eqref{eq:ConditionOR} into direct ``order conditions'' on the scheme coefficients. By expanding terms like $z \vecpower{b}{T}(I -z A)\vec{\tau}^{(k)}$ in the variable $H = z/(1-z\gamma)$ (which is possible only for SDIRK methods), Scholz developed B-convergence order barriers (up to order $4$) for ROW methods and demonstrated that the weaker conditions are compatible with ROW (SDIRK) structures by constructing schemes with up to $p=q=4$ in \eqref{eq:ConditionOR}. The complexity of Scholz's conditions for ROW methods were subsequently simplified by Ostermann \& Roche \cite[Eq.~(3.12')]{OstermannRoche1993} to (effectively) the conditions \eqref{eq:OR_WSO}, where the $A$ and $\vec{\tau}^{(k)}$ in \eqref{eq:OR_WSO} are ROW analogues of the corresponding RK scheme quantities.
For further discussion of order reduction in the
context of stiff problems we refer to
\cite[Chapter~4]{biswas2021structure}, \cite[Chapter~7]{dekker1984} and \cite[Chapter~IV.15]{wanner1996solving} and references therein.

Although originally derived only for SDIRK methods, the ``order conditions'' \eqref{eq:OR_WSO} are in fact necessary conditions for any RK scheme to avoid order reduction in \eqref{eq:pro_robin_stiff_problem}; see \S\ref{subsec:WSO} and \cite[Eqs.~(20)--(21)]{rang2014analysis}, \cite[Thm.~3]{Rang2016}, \cite{skvortsov2017avoid} and also \cite{rosales2017spatial, KetchesonSeiboldShirokoffZhou2020}. A geometric interpretation of \eqref{eq:OR_WSO} is provided in \cite{rosales2017spatial}. Recently, \cite[Chap.~6]{Roberts2021} and \cite{RobertsSandu2022} extended \eqref{eq:OR_WSO} to generalized-structure additively partitioned Runge-Kutta (GARK) methods.

\section{Bounds on Weak Stage Order via $A$-Invariant Subspaces}
\label{sec:WSO_Recast_As_OrthSpaces}

We now state the main results of the paper.
\begin{theorem}\label{Thm:MainResult1}(Main Result)
A Runge-Kutta scheme with $s$ stages, $\numc$ distinct abscissas, order $p$, and weak stage order $q$ satisfies:
\begin{enumerate}
    \item If the abscissa values are all non-zero ($c_j \neq 0$), then $q\le 2\numc -1$.
    \item If some of the abscissa values are zero ($c_j = 0$), then either (i)~$q\le 2\numc -2$, or (ii)~$q = \infty$ in which case $p = 1$. 
    \item If $q\le 2\numc -1$, then\vspace{-.5em}
    \begin{equation*} 
        q + \left\lfloor\frac{p+1 + \sigma}{2}\right\rfloor \leq s + \numc\;,
    \end{equation*}
    where $\sigma = 1$ for stiffly accurate methods with invertible $A$, and $\sigma = 0$ otherwise.
\end{enumerate}
\end{theorem}
A stricter bound can be obtained for DIRK schemes.
\begin{theorem}\label{Thm:MainResultDIRK} (Main Result for DIRK)
Let a DIRK scheme be given with $s$ stages, $\numc$ distinct abscissas, order $p$, and weak stage order $q\leq 2 \numc - 1$.  Then
\begin{equation*}
    \left\lfloor \frac{q + \kappa}{2}\right\rfloor - \kappa + p \leq s + 1 - \sigma\;,
\end{equation*}
where $\sigma = 1$ if $A$ is invertible and the scheme is stiffly accurate, and $\sigma = 0$ otherwise; and $\kappa = 1$ if $A$ is a GEDIRK scheme, and $\kappa = 0$ otherwise. 
\end{theorem}
Here $\lfloor x \rfloor$ is the standard floor function for the real number $x$. Note that \Cref{Thm:MainResult1} still applies when $q \geq 2 \numc - 1$.

\begin{remark}(Order barriers) When $p = s + 1 - \sigma$, \Cref{Thm:MainResultDIRK} implies that DIRK schemes are limited to WSO $q \leq 3$, with the exception of GEDIRK schemes which are limited to WSO $q \leq 4$. For a fixed $s$, decreasing $p$ by $1$ increases the upper bound on $q$ by $2$. There exist schemes that satisfy these barriers sharply (see \S\ref{sec:Examples}). \myremarkend
\end{remark}

In fact, the main results appearing here can be viewed as improvements to the classical bounds on the RK order $p$ via the number of stages $s$. For instance, the classical bound $p \leq s + 1$ for DIRK methods can be recovered by setting $q = 0$ ($\kappa = 0$) in \Cref{Thm:MainResultDIRK}.   Moreover, the main results show that the maximum WSO $q$ is controlled by the difference in the classical bound (e.g., $s+1 - p$ for DIRKs). 

\subsection{A Pair of Orthogonal $A$-invariant Subspaces}
\label{Subsec:OrthSpaces}
The WSO conditions \eqref{eq:OR_WSO} can be written equivalently in terms of orthogonal subspaces.  Given an $s$-stage RK method, let
\begin{align}\label{Def:Kspace}
    K_\kgen &:=\operatorname{span}\left\{\vec{\tau}^{(1)}, A \vec{\tau}^{(1)}, \ldots, A^{s-1} \vec{\tau}^{(1)}, \vec{\tau}^{(2)}, A \vec{\tau}^{(2)}, \ldots, A^{s-1} \vec{\tau}^{({\kgen})}\right\}\; \quad (m \geq 1) \;,  \\
    \label{Def:Yspace}
    Y &:=\operatorname{span}\left\{\vec{b}, A^T\vec{b}, \ldots, (A^T)^{s-1} \vec{b}\right\}.
\end{align}
Note that due to the Cayley-Hamilton theorem, $K_\kgen$ is an $A$-invariant subspace, while $Y$ is left $A$-invariant, i.e., $A \vec{v} \in K_\kgen$ for any $\vec{v} \in K_\kgen$ and $A^T \vec{v} \in Y$ for all $\vec{v} \in Y$. In fact, $K_m$ and $Y$ are the smallest $A$-invariant spaces containing the $\vec{\tau}^{(j)}$ ($j \leq m$) and $\vec{b}$ respectively. This allows for the natural generalization to $\kgen = \infty$ in \eqref{Def:Kspace} by defining $K_\infty$ as the smallest $A$-invariant space containing $\tau^{(k)}$ for all $k \geq 1$. 

Then \eqref{eq:OR_WSO} is equivalent to the condition that $\vecpower{v}{T} \vec{w} = 0$ for every $\vec{v} \in Y$ and $\vec{w} \in K_q$.  Since $K_m$ and $Y$ are subspaces of $\mathbb{R}^s$, we have
\begin{lemma}\label{Lem:WSODefOrthSubspaces} 
Given a Runge-Kutta scheme defined by $A, \vec{b}$, let $q \geq 1$ be an integer or $\infty$, and $K_q$ and $Y$ be the subspaces \eqref{Def:Kspace} and \eqref{Def:Yspace} of $\mathbb{R}^s$.  Then the scheme has weak stage order of at least $q$ if and only if $Y$ and $K_q$ are orthogonal, in which case
\begin{equation}\label{Dim:inequality}
    \operatorname{dim}(Y)+\operatorname{dim}(K_q) \leq s\;.
\end{equation}
\end{lemma}

RK methods with (classical) stage order $q$ have $\dim(K_q)=0$ and $\dim(K_{q+1}) > 0$. Except for schemes that were specifically designed with high WSO, almost all irreducible schemes in the literature have $\dim(Y)=s$, so Lemma~\ref{Lem:WSODefOrthSubspaces} implies that their WSO is no greater than their stage order. However, some existing schemes have $\dim(Y)<s$; for example, explicit methods with one or more entries on the first subdiagonal equal to zero. These include the 8th-order method of Prince \& Dormand \cite{prince1981}, as well as some of the fifth-order SSP methods of Ruuth \& Spiteri \cite{ruuth2004}.

\begin{remark} In the proof of \Cref{Thm:MainResultDIRK} (to follow) the bound \eqref{thm42boundb} for DIRK schemes is used.  However, if $p = 1$, \eqref{thm42bounda}  for fully implicit schemes gives $\dim Y \geq 1$ which is tighter than the DIRK bound \eqref{thm42boundb} ($\dim Y \geq 0$) and yields a slight improvement to \Cref{Thm:MainResultDIRK}. \myremarkend 
\end{remark}

\begin{remark}
For simplicity we have written ``order p'' in \Cref{Thm:MainResult1} and  \Cref{Thm:MainResultDIRK}; however, both results hold under the weaker condition that $R(z) = e^{z} + \BigO(z^{p+1})$ is a $p$-th order approximation as $z\rightarrow 0$.
\myremarkend
\end{remark}

The proof of the main results follows from the inequality \eqref{Dim:inequality} along with lower bounds on the dimensions of $Y$ and $K_q$. The proof can be divided into three steps. Step 1 (\S\ref{Subsec:ArbSchemes_YandQ}) bounds the dimension of $Y$ for both fully implicit and DIRK schemes; Step 2 (\S\ref{Subsec:ArbSchemes_Kq}) bounds the dimension of $K_\kgen$ for fully implicit schemes which is then refined in Step 3 (\S\ref{Subsec:KandP}) for DIRK schemes.

\subsection{(Step 1) A Lower Bound on the Dimension of the Subspace \texorpdfstring{$Y$}{Y}}
\label{Subsec:ArbSchemes_YandQ} 
Here we show that the degrees of the numerator and the denominator of the Runge-Kutta stability function are bound by the dimension of the subspace $Y$. When combined with classical results on rational approximations to $e^z$, we obtain a lower bound on $\dim Y$ in terms of $p$.

Throughout, we denote a rational function that approximates $e^z$ to order $p$ as
\begin{align}\label{Eq:RationalApprox}
    N(z)/D(z) = e^z + \mathcal{O}(z^{p+1})\;, \quad \textrm{as} \; z \rightarrow 0\;,
\end{align}
where $N$ and $D$ are polynomials with no common factors.  

\begin{theorem}(Rational approximations of $e^z$, \cite[Thm.~IV.3.11,Thm.~IV.4.18]{wanner1996solving}, \cite[Thm.~2.1]{NorsettWolfbrandt1977})\label{Thm:RealPoles} 
    Let polynomials $N(z), D(z)$ satisfy \eqref{Eq:RationalApprox}.
    Then $p \leq \deg N + \deg D$. If, further, $N/D$ has only real poles then $p \leq \deg N + 1$.
\end{theorem}

For an $s$-stage RK method, the stability function is a rational approximation of $e^z$ with numerator and denominator of degree at most $s$.  In fact, we can replace $s$ in this statement with $\dim Y$.
\begin{theorem} \label{thm:rational-degree}
    Let a RK method $(A,b)$ be given, with $Y$ defined
    as in \eqref{Def:Yspace}.  Then the stability
    function of the method is $R(z)=N(z)/D(z)$,
    where $\deg N \le \dim Y$ and $\deg D \le \dim Y$.
\end{theorem}
\begin{proof}
Define an orthogonal matrix $\basismat = ( \Umat \; | \; \Vmat) \in \mathbb{R}^{s\times s}$, where the columns of $\Umat \in \mathbb{R}^{s \times d}$, with $d:= \dim Y$, are a basis of $Y$ and the columns of $\Vmat$ are a basis of its orthogonal complement.
Then $\vecpower{b}{T} V = \vec{0}$ and $\vecpower{b}{T}(I-zA)^{-1}V = \vec{0}$.

So the stability function \eqref{eq:stabilityfunction1} is
\begin{equation}\label{Eq:StabFuncdimY} 
    R(z) = 1 + z\vecpower{b}{T}  \basismat \basismat^T (I - z A)^{-1} \basismat \basismat^T \vec{e} = 1 + z \vecpower{b}{T}_U  M \Umat^T \vec{e}\;,
\end{equation}
where $\vecpower{b}{T}_U = \vecpower{b}{T}\Umat \in \mathbb{R}^d$ and $M = \Umat^T (I-zA)^{-1} \Umat \in \mathbb{R}^{d\times d}$. Using this and Cramer's rule it can be shown that the numerator and denominator of $R(z)$ are given by determinants of $d\times d$ matrices whose entries are linear functions of $z$. Thus $R(z)$ is a rational function with numerator and denominator of degree at most $d$.
\end{proof}
For stiffly accurate schemes with $A$ invertible, $R(z)$ is further constrained to satisfy $\lim_{z \rightarrow \infty} R(z) = \lim_{z\rightarrow \infty} N(z)/D(z) = 0$, so that $\deg N \leq \deg D - 1 \leq  \dim Y - 1 $.  For DIRK schemes, the stability function $R(z)$ has real poles since the eigenvalues of $A$ are confined to the diagonal entries---which are real.  Combining these observations with \Cref{Thm:RealPoles} and \Cref{thm:rational-degree}, we have proved:

\begin{theorem}\label{Thm:BoundOnDegR_ForDIRK}
    Let a RK scheme $(A, b)$ be given with $Y$ defined as in \eqref{Def:Yspace} and stability function $R(z) = e^{z} + \BigO(z^{p+1})$, as $z \rightarrow 0$.  Then 
     \begin{equation} \label{thm42bounda}
        \left\lfloor \frac{p+1 + \sigma}{2} \right \rfloor \leq \dim Y\;.
    \end{equation}
    If the scheme is diagonally implicit, \eqref{thm42bounda} we
    have the stronger bound
    \begin{equation} \label{thm42boundb}
        p \leq \dim(Y) + 1 - \sigma\;.
    \end{equation}
    where $\sigma = 1$ if the scheme is stiffly accurate and $A$ is invertible, and $\sigma = 0$ otherwise.
\end{theorem}

\subsection{(Step 2) A Lower Bound on the Dimension of the Subspace \texorpdfstring{$K_\kgen$}{K} for General RK Schemes}\label{Subsec:ArbSchemes_Kq}
Here we provide a lower bound on $\dim(K_\kgen)$ in terms of $\kgen$ and the number of distinct abscissas $\numc$. The key idea is that $K_\kgen$ contains linear combinations of the vectors $\vec{\tau}^{(k)}$, and for large enough $\kgen$ the addition of each new vector increases the dimension of the space $K_\kgen$.  

We first denote the column space of the Vandermonde matrix with columns $\vec{c}$
\begin{align*}
    V := 
     \left\{ 
        \begin{array}{ll}
        \textrm{span}\{ \vec{e}, \vec{c}, \ldots , \vecpower{c}{n_c - 1} \}, & \textrm{if all } c_j \neq 0 \;, \\
        \textrm{span}\{ \vec{c}, \vecpower{c}{2}, \ldots , \vecpower{c}{n_c - 1} \}, & \textrm{if at least one } c_j = 0,
        \end{array}
        \right. .
\end{align*}

Now define subspaces of $V$ for values of $\kgen \geq \numc + 1$ as follows:
\begin{align*} 
    W_\kgen &:= \left\{ w(C) \vec{e} \; \big| \; \textrm{for all} \; w(x) = \int_0^x \alpha(x) \minPolc(x) \; {\rm dx}
    \;\; \textrm{and} \;\; \alpha(x) \in \Pi_{\kgen - 1 - \numc}\right \}, \\
    \Pi_d &= \left\{
        \alpha_0 + \ldots + \alpha_{d} x^d \; | \; 
        \alpha_j \in \mathbb{R}, \textrm{ for } j = 0, \ldots, d
    \right\}.
\end{align*}
Here $\minPolc(x)$ is the minimal polynomial for the matrix $C$ (so that $\minPolc(C) = 0$) and $\Pi_d$ is the space of polynomials of degree at most $d$. By construction $W_m \subseteq W_{m+1}$. 

We denote the smallest $A$-invariant subspace containing a vector space $X$ by $\mathcal{K}(X)$; namely $\mathcal{K}(X)$ can be constructed explicitly by taking the span of all the $A$-Krylov subspaces of the basis vectors of $X$ (similar to \eqref{Def:Kspace}). In particular, $\mathcal{K}(V)$ and $\mathcal{K}(W_m)$ are the smallest $A$-invariant spaces containing $V$ and $W_m$ respectively. With these notations, we now have

\begin{lemma}(The $K_\kgen$ Sandwich) 
For any $\kgen \geq \numc + 1$ (or $m = \infty$) it holds:
 \begin{align*}
      W_m \subseteq  \mathcal{K}(W_m) \subseteq K_\kgen \subseteq \mathcal{K}(V)\;.
 \end{align*}
\end{lemma}
\begin{proof}
    The fact that $K_\kgen \subseteq \mathcal{K}(V)$ follows from the observation that every $\vec{\tau}^{(k)}$, for arbitrary $k$, is a linear combination of $A \vecpower{c}{k} \in \mathcal{K}(V)$ and $\vecpower{c}{k+1} \in \mathcal{K}(V)$.
    
    Next we show that $K_\kgen$ contains the subspace $W_m$. Due to the fact that $\mathcal{K}(W_m)$ and $K_\kgen$ are both A-invariant, this implies $\mathcal{K}(W_m) \subseteq K_\kgen$. 
    
    Let $w'(x) = w_{\kgen-1} x^{\kgen-1} + w_{\kgen-2} x^{\kgen-2} + \ldots + w_0$ denote an arbitrary (real) polynomial and let $w(x) := \int_0^x w'(s) \; {\rm ds}$. The space $K_\kgen$ then contains the linear combination
    \begin{align}
        \nonumber
        \sum_{j=1}^{\kgen} w_{j-1} \vec{\tau}^{(j)} &=  \left(\sum_{j=1}^{\kgen} A w_{j-1} C^{j-1} - \frac{1}{j} w_{j-1} C^j \right) \vec{e}, \quad \textrm{using \eqref{Eq:SOV_MatrixForm}}\;, \\ \label{Eq:VectinKq}
        &= \big( A w'(C) - w(C) \big) \vec{e} \in K_\kgen\;.
    \end{align}
    Since $C$ is diagonal, its minimal polynomial $\minPolc(x)$ has simple roots (one for each distinct entry of $\vec{c}$) and so $\deg \minPolc = \numc$. Thus, for $\kgen \geq \numc + 1$, we can set $w'(x) = \alpha(x) \minPolc(x)$ in \eqref{Eq:VectinKq}, (so that $w'(C) = 0$) showing that $w(C)\vec{e} \in K_\kgen$.  Hence $K_\kgen$ contains the subspace $W_m$.    
\end{proof}
The next theorem follows directly from properties of $W_m$. 
\begin{theorem}(Properties of $K_\kgen$)\label{thm:GenPropertiesKq}
Consider a Runge-Kutta scheme with $\numc$ distinct abscissas and the space $K_\kgen$ defined in \eqref{Def:Kspace}. Then:
    \begin{enumerate}
        \item[(a)] If $\kgen \leq 2\numc - 1$, then $\dim K_\kgen \geq \max\{\kgen- \numc, 0\}$. 
        \item[(b)] If $\kgen \geq 2\numc$ and all $c_j \neq 0$ (i.e., $\minPolc(0) \neq 0$), then $K_\kgen = \mathcal{K}(V)$; in particular $\vec{e} \in K_\kgen$ and $\dim K_\kgen \geq \numc$.
        \item[(c)] If $\kgen \geq 2\numc - 1$ and $c_j = 0$ for at least one value of $j$ (i.e., $\minPolc(0) = 0$), 
        then $K_\kgen = \mathcal{K}(V)$, in particular $\vec{c} \in K_\kgen$ and $\dim K_\kgen \geq \numc - 1$;
    \end{enumerate}
\end{theorem}

\begin{proof}
    For (a), we compute $\dim W_m$ which bounds $\dim K_\kgen$ from below.
    First note that vectors of the form $w(C) \vec{e}$ (for any polynomial $w(x)$) are in one-to-one correspondence with their associated polynomial remainders\footnote{Here $w(x) \mod \minPolc(x)$ is the remainder of $w(x)$ when divided by $\minPolc(x)$; the set of all such remainders can be identified with the vector space $\Pi_{\numc - 1}$.}
    $w(x) \mod \minPolc(x)$ (i.e., as vector spaces they are isomorphic\footnote{For any polynomial $w(x) \mod \minPolc(x)$, the identification of $w(x) \mapsto w(C)\vec{e}$ is surjective and preserves the operations of the vector space. It is also injective: since $w(C)$ is a diagonal matrix, the vector $w(C) \vec{e} = 0$ if and only if $w(C) = 0$, and $w(C) = 0$ if and only if $w(x)$ is divisible by $\minPolc(x)$. Hence, the vectors $w_1(C) \vec{e} = w_2(C) \vec{e}$ if and only if $w_1(x) - w_2(x) = \alpha(x) \minPolc(x)$ for some $\alpha(x)$, i.e., $w_1(x) \equiv w_2(x) \mod \minPolc(x)$.}). Thus, $W_m$ has the same dimension as the range of $I : \Pi_{\kgen - 1 - \numc} \rightarrow \Pi_{\numc - 1}$ defined by
    \begin{equation*}
        I[\alpha] := \left( \int_0^x \alpha(x) \minPolc(x) \; {\rm dx} \right) \mod \minPolc(x)\;,
    \end{equation*} 
    where $I[\cdot]$
    is the linear (integral) operator that maps $\Pi_{\kgen - 1 - \numc}$ (which has dimension $\kgen - \numc$) into polynomials modulo $\minPolc(x)$ (which has dimension $\numc$).  
    We now claim that $I[\cdot]$ has a trivial null space whenever $\kgen \leq 2 \numc - 1$. Suppose not---then there are polynomials $\alpha(x) \neq 0$ ($\deg \alpha \leq \numc - 2$) and $\beta(x)$ ($\deg \beta \leq \numc -1$) such that
    \begin{equation*}
        \int_0^x \alpha(x) \minPolc(x) {\rm dx} = \beta(x) \minPolc(x)\;.
    \end{equation*}
    Differentiating both sides yields:
    \begin{equation}\label{Eq:DiffNull}
        \left( \alpha(x) - \beta'(x) \right) \minPolc(x) = \beta(x) \minPolc'(x)\;.
    \end{equation}
    The minimal polynomial $\minPolc(x)$ contains $\numc \geq 1$ distinct roots, and hence $\minPolc'(x)$ has no common root with $\minPolc(x)$ (i.e., $\minPolc(x)$ and $\minPolc'(x)$ are relatively prime). Thus, \eqref{Eq:DiffNull} requires that 
    $\minPolc'(x)$ divide $\left( \alpha(x) - \beta'(x) \right)$ and $\minPolc(x)$ divide $\beta(x)$, which is impossible in light of the degrees of $\alpha(x)$, $\beta(x)$. Thus no $\alpha(x)$ exists, and whenever $\numc + 1 \leq \kgen \leq 2\numc - 1$ one has the following dimension:
    \begin{equation}\label{eq:lb_Kq}
        \dim W_m = \dim \Pi_{\kgen- 1- \numc} = (\kgen - \numc) \leq \dim K_\kgen\;.
    \end{equation}
    Combining \eqref{eq:lb_Kq} with the trivial $\dim K_\kgen \geq 0$ when $\kgen < \numc + 1$ proves (a). 
    
    For (b), when $\kgen = 2\numc$, \eqref{Eq:DiffNull} has the (unique) family of solutions
    \begin{align*}
        \beta(x) = \gamma \minPolc(x) \;, \quad \textrm{and} \quad 
        \alpha(x) = 2 \gamma \minPolc'(x) \;  \quad (\gamma \in \mathbb{R}) \;.
    \end{align*}
    Substituting this value of $\alpha(x)$ into 
    \begin{align*}
        I[\alpha] = \gamma\minPolc(x)^2 - \gamma \minPolc(0)^2 = - \gamma \minPolc(0)^2  \mod \minPolc(x) \; ,
    \end{align*}    
    shows that $\gamma = 0$ is still the only solution to $I[\alpha] = 0$. Hence, $I[\cdot]$ has a trivial null space, $W_m$ contains the vector $\vec{e}$ (e.g., take $\gamma = -\minPolc(0)^{-2}$), and $\dim W_m = \numc$.  Thus, $W_m$ has the same dimension as $V$ which forces $W_m = V$ showing that $K_\kgen = \mathcal{K}(V)$. For $m \geq 2 \numc$, $K_\kgen = \mathcal{K}(V)$ holds trivially since $K_m \subseteq K_{m+1}$. 
    
    For (c), if $\minPolc(0) = 0$, then every $w(x)$ appearing in $W_m$ is divisible by $x^2$ so that $w(x)$ and $\minPolc(x)$ share a common factor of $x$. Thus $x$ divides the remainder $w(x) \mod \minPolc(x)$ as well as every polynomial in the range of $I[\cdot]$, i.e., for every $\alpha \in \Pi_{\kgen-1-\numc}$, $I[\alpha] = x r(x)$ for some $r(x)$ with $\deg r \leq \numc - 2$. When $\kgen = 2 \numc - 1$ the range of $I[\cdot]$ has dimension $\numc - 1$ and includes all monomials $x^{j}$ for $1 \leq j \leq \numc - 1$. Hence $W_m = V$ showing that $K_\kgen = \mathcal{K}(V)$. 
    
    The vector $\vec{c}$ has an explicitly construction: since $\minPolc(x)$ and $\minPolc'(x)$ are relatively prime, polynomials $\beta(x)$ ($\deg \beta \leq \numc - 1$) and $\gamma(x)$ ($\deg \gamma(x) \leq \numc - 2$) exist such that
    \begin{equation*}
    \gamma(x) \minPolc(x) = \beta(x) \minPolc'(x) + 1\;.
    \end{equation*}
    Set $\alpha(x) := \gamma(x) - \beta'(x)$ ($\deg \alpha \leq \numc - 2$), so that $\alpha(x) \minPolc(x) = \frac{d}{dx}\left( \beta(x) \minPolc(x) \right) + 1$. Since $\minPolc(0) = 0$, integrating both sides shows that $I[\alpha] = x$ and hence: $\vec{c} \in W_m$.      
\end{proof}

By combining the lower bounds, we obtain a proof of the main result.
\begin{proof}(Of the Main Result \Cref{Thm:MainResultDIRK}) 
    Parts (1) and (2) of \Cref{Thm:MainResultDIRK}     
    follow from \Cref{thm:GenPropertiesKq}(b) and (c) (respectively) with the fact that the order conditions $\vecpower{b}{T} \vec{e} = 1$ (for $p = 1$) and $\vecpower{b}{T} \vec{c} = 1/2$ (for $p = 2$) are incompatible with the WSO requirement that $\vec{b} \perp K_\kgen$ whenever $\vec{e} \in K_\kgen$ or $\vec{c} \in K_\kgen$.
    
    \Cref{Thm:MainResult1}(3) follows from combining the inequality \eqref{Dim:inequality}; the lower bound for $\dim Y$ from \Cref{Thm:BoundOnDegR_ForDIRK} for fully implicit schemes; and the lower bound for $\dim K_m$  in \Cref{thm:GenPropertiesKq}(a).
\end{proof}

%
\subsection{(Step 3) A Lower Bound on the Dimension of the Subspace \texorpdfstring{$K_\kgen$}{K} for DIRK Schemes}
\label{Subsec:KandP}
We now turn to obtaining lower bounds on the dimension of $K_\kgen$ in the more restrictive setting when $A$ is DIRK. The results in this section build on \Cref{thm:GenPropertiesKq}. The difference of $1$ (dimension) in parts (b) and (c) in \Cref{thm:GenPropertiesKq} creates a somewhat bothersome issue for directly obtaining the lower bound in this section. However, this difference of $1$ can be avoided for DIRK schemes that do not admit a GEDIRK structure.
\begin{lemma}[Extra dimension in $K_m$ when $c_j = 0$]\label{lem:Refinementcj_zero} 
    For a DIRK scheme that is not a GEDIRK scheme, if $\kgen \geq 2\numc$ then $\dim K_\kgen \geq \numc$ (where the space $K_\kgen$ is defined in \eqref{Def:Kspace}).
\end{lemma}
\begin{proof}
    If $\vec{c}$ has no zero, then \Cref{thm:GenPropertiesKq}(b) applies and we are done. 
    Let $\ell$ be the first occurrence of $c_\ell =0$ in $\vec{c}$ (i.e., $c_j \neq 0$ for $j < \ell$). Note that $\ell > 1$ since $c_1 = a_{11} \neq 0$ (because $A$ is not a GEDIRK). \Cref{thm:GenPropertiesKq}(c) shows that $W_m \subseteq K_\kgen$ has dimension $\numc - 1$, that the vector $\vec{c} \in K_\kgen$, and that $W_m$ is orthogonal to $\vec{e}_{\ell}$, i.e., $\vecpower{e}{T}_{\ell} \vec{u} = 0$ for all $\vec{u} \in W_m$. In addition, $W_m$ contains a vector of the form (since $\ell$ is the first occurrence of $0$)
     \begin{equation*}
        \vec{y} := (1, \ldots, 1, 0, \star ) \in W_m\;,
    \end{equation*}
    where the first $\ell - 1$ entries of $\vec{y}$ are $1$ and $\star$ is unimportant. Since $K_\kgen$ is $A$-invariant, we have $\vec{c} - A \vec{y} \in K_\kgen$; however $\vecpower{e}{T}_{\ell} (\vec{c} - A \vec{y}) = a_{\ell \ell} \neq 0$ (where $a_{\ell \ell} \neq 0$ since $A$ is not a GEDIRK), showing that $\dim K_\kgen \geq \numc$.
\end{proof}

To expedite the remaining proofs below, we define a truncation map $\lbmap \cdot \rbmap_{j}$ which acts on matrices as $\lbmap \cdot \rbmap_{j}:\mathbb{R}^{s\times s} \rightarrow \mathbb{R}^{j \times j}$ to isolate the upper $j \times j$ block of $A$, and acts on vectors as $\lbmap \cdot \rbmap_{j}:\mathbb{R}^{s} \rightarrow \mathbb{R}^{j}$ to isolate the top $s$ components of $\vec{c}$:
\begin{equation*}
\lbmap A\rbmap_{j} :=\begin{pmatrix}
{a_{11}} & {a_{12}} & {\cdots} & {a_{1j}} \\
{a_{21}} & {\ddots} &     {}   & {\vdots} \\
{\vdots} &     {}   &     {}   &   {} \\
{a_{j1}} & {a_{j2}} & {\cdots} & {a_{jj}}
\end{pmatrix}
\ \ \text{and} \ \
\lbmap \vec{c}\rbmap_{j}:=
\begin{pmatrix}
{c_{1}} \\
{c_{2}}  \\
{\vdots} \\
{c_{j}} 
\end{pmatrix}.
\end{equation*}
For a lower-triangular matrix $A$, the first $j$ ($\leq s$) components of the vectors $\vec{\tau}^{(k)}$ and $\vec{c}$ are functions of the upper $j \times j$ sub-block of the matrix $A \in \mathbb{R}^{s\times s}$ only. Moreover:
\begin{equation}\label{Eq:CommutationProperty}
    \lbmap A^n\rbmap_{j} = \left( \lbmap A \rbmap_{j}\right)^{n},
    \quad \textrm{and} \quad
    \lbmap \vecpower{c}{n}\rbmap_{j}=\left(\lbmap\vec{c}\rbmap_{j}\right)^{n}.
\end{equation}
We will also refer to the vectors $\vec{\tau}^{(k)}$ 
that arise when $A$ is replaced by $\lbmap A\rbmap_j$
in \eqref{Eq:SOV_MatrixForm}:
\begin{align}
    \vec{\tau}^{(k)}(\lbmap A\rbmap_j) &:= \lbmap A\rbmap_j \lbmap \vecpower{c}{k-1}\rbmap_j -\tfrac{1}{k}\lbmap \vecpower{c}{k}\rbmap_j
    \quad 
    \textrm{for } k \geq 1\;,
\end{align}
which commute with the map in the sense that for every $1\leq j \leq s$, $n \geq 0$, and $k \geq 1$:
\begin{equation}\label{Eq:CommutationTau}
    \lbmap \vec{\tau}^{(k)}(A) \rbmap_j = \vec{\tau}^{(k)}( \lbmap A \rbmap_j )\;, \quad
    \lbmap A^n \vec{\tau}^{(k)}(A) \rbmap_j = \lbmap A \rbmap_j^n \vec{\tau}^{(k)}( \lbmap A \rbmap_j )\;.
\end{equation}
We also extend $\lbmap \cdot \rbmap$ to vector subspaces $U \subseteq \mathbb{R}^s$ as $\lbmap U \rbmap_j := \textrm{span}\{ \lbmap \vec{u} \rbmap_j \; | \; \vec{u} \in U\} $, which implies $\dim U \geq \dim \lbmap U \rbmap_j$.

\begin{lemma}\label{Lem:PropKqDIRK} (Dimension of $K_\kgen$ for a DIRK scheme)
    Consider a DIRK scheme with $\numc$ distinct abscissas, and the corresponding space $K_\kgen$ defined in \eqref{Def:Kspace}. 
    Then
    \begin{equation}\label{Eq:LB_Kq_DIRK}
        \dim K_\kgen \geq \min\left\{ \left \lfloor \frac{m + \kappa}{2} \right\rfloor, n_c \right\} - \kappa\;,
    \end{equation}
    where $\kappa = 1$ if $A$ is a GEDIRK scheme, and $\kappa = 0$ otherwise. 
\end{lemma}
The main idea in the proof is that the commutation property \eqref{Eq:CommutationTau} implies that the matrix $\lbmap A \rbmap_j$ and vectors $\vec{\tau}^{(k)}( \lbmap A \rbmap_j )$, define an $\lbmap A \rbmap_j$-invariant space. Specifically, $\lbmap K_\kgen \rbmap_j$ is exactly the smallest $\lbmap A \rbmap_j$-invariant space containing $\vec{\tau}^{(k)}( \lbmap A \rbmap_j )$ for $k=2, \ldots, \kgen$. Hence, \Cref{thm:GenPropertiesKq} and \Cref{lem:Refinementcj_zero} apply\footnote{The new space $W_m$ as well as the definition $I[\cdot]$ must use the minimal polynomial of $\lbmap C \rbmap_j$.} to $\lbmap A \rbmap_j$ for each $1 \leq j \leq s$, where $\numc$ is replaced with the number of distinct values in $\{c_1, \ldots, c_j\}$. The dimension of the associated $\lbmap A \rbmap_j$-invariant space (for any $j$)  bounds $\dim K_\kgen \geq \dim \lbmap K_\kgen \rbmap_j$. We then pick the ``worst case'' $j$. For simplicity let
\begin{equation*}
    T(j) := \max \{ r \; | \: \# \{ c_1, \ldots, c_r\} = j \}\;,
\end{equation*}
where $\#$ denotes the number of distinct values in a set. Then $T(\cdot)$ is strictly increasing, $T(1) \geq 1$, $T(\numc) = s$, and the set $\{c_{1}, \ldots, c_{T(j)} \}$ contains $j$ distinct values.
\begin{proof}(of \Cref{Lem:PropKqDIRK})
    Assume $\kgen > 1$ (it holds trivially for $\kgen = 1$). If $A$ is not a GEDIRK, set $r := \min\{ \lfloor \frac{\kgen}{2} \rfloor, \numc\}$. Apply \Cref{lem:Refinementcj_zero} to $\lbmap A \rbmap_{T(r)}$, which has $r$ distinct $\vec{c}$ values, to obtain $\dim K_\kgen \geq r$ (which is \eqref{Eq:LB_Kq_DIRK} when $\kappa = 0$). 
    
    If $A$ is a GEDIRK we use the more general bound \Cref{thm:GenPropertiesKq}(a--c) applied to a DIRK matrix $\lbmap A \rbmap_{T(r')}$ with $r'$ ($\leq m$) distinct abscissas values:
    \begin{align}\label{Eq:GenDIRK_Kq_bnd}
        \dim K_{m} \geq \min\{ r' - 1, m - r' \}. 
    \end{align}
    Set $r' := \min\{ \frac{\kgen}{2}, \numc\}$ if $m$ is even and $r' := \min\{ \frac{\kgen + 1}{2}, \numc\}$ if $m$ is odd. Applying  \eqref{Eq:GenDIRK_Kq_bnd}, and observing that (in all 4 cases) $\min\{ r' - 1, m - r' \} = r' - 1$ yields \eqref{Eq:LB_Kq_DIRK} with $\kappa = 1$. 
\end{proof}

We conclude this section with a proof of the main result for DIRKs.
\begin{proof}(Of the Main Result for DIRKs, \Cref{Thm:MainResultDIRK})
    When $q \geq 2 \numc - 1$, the ~\Cref{Lem:PropKqDIRK} gives a lower bound on $\dim K_m$ of $\lfloor \frac{m + \kappa}{2} \rfloor - \kappa$. Substituting this and the lower bound on $\dim Y$ in \Cref{Thm:BoundOnDegR_ForDIRK} for DIRKs into \eqref{Dim:inequality} yields the result.        
\end{proof}

\section{Minimal Polynomials for the Spaces \texorpdfstring{$Y$}{Y} and \texorpdfstring{$K_q$}{K}}
\label{Sec:MinPoly}
Our results in the next two sections will be based on the linear algebra of orthogonal left and right invariant subspaces applied to $Y$ and $K_\kgen$. We therefore first review and establish key results (e.g.,  eigenspace, eigenvalue and minimal polynomials) pertaining to the action of a matrix on its orthogonal invariant subspaces. We employ the standard terminology: a \emph{monic polynomial} $p(x)$ is defined as having a leading (i.e., highest degree term) coefficient of one; the \emph{characteristic polynomial} of a matrix is defined as the (monic) polynomial $\textrm{char}_A(x) := \det(xI - A)$; the \emph{minimal polynomial} $p(x)$ of a matrix $A$ is the monic polynomial of smallest degree for which $p(A) = 0$.  While it is often the case that (for instance when the eigenvalues of $A$ are distinct) the minimal polynomial is the characteristic polynomial, i.e., $p(x) = \textrm{char}_A(x)$, in general $p(x)$ is of lower degree than $\textrm{char}_A(x)$ when $A$ has repeated eigenvalues.

For a real matrix $A$, \cite[Chap.~8 \& 9A]{Axler2015} the minimal polynomial $p(x)$ of $A$: 
(i)~is unique;
(ii)~has real coefficients\footnote{More precisely, when $A$ is real with $p_r(x)$ the minimal polynomial with coefficients in $\mathbb{R}$, and $p_c(x)$ is the minimal polynomial with coefficients over $\mathbb{C}$, then $p_r(x) = p_c(x)$.};
(iii)~divides the characteristic polynomial of A (thus every root of $p(x)$ is an eigenvalue of $A$ and the $\textrm{deg} \; p \leq \textrm{deg} \; \textrm{char}_A(x)$); moreover, 
(iv)~every eigenvalue of $A$ (not including multiplicity) is a root of $p(x)$;
(v)~every polynomial $\tilde{p}(x)$ satisfying $\tilde{p}(A) = 0$ is divisible by $p(x)$.
The following two theorems concern minimal polynomials for matrices restricted to invariant subspaces and will be employed in our study of WSO.
\begin{theorem}\label{Thm:AinvariantSpaces} ($A$-invariant subspaces)
    Let $A \in \mathbb{R}^{s\times s}$, and $U \subseteq \mathbb{R}^s$ an $A$-invariant subspace with dimension $d := \operatorname{dim}(U)$ ($d=0$ is possible). Then there exists a unique monic polynomial $p(x)$ of minimal degree and coefficients in $\mathbb{R}$, such that:
    \begin{align}\label{def:monicpoly}
        p(A) \vec{u} = 0, \quad \forall \vec{u} \in U\;.
    \end{align}
    This polynomial has the following properties:
    \begin{enumerate}
        \item[(a)] $p(x)$ divides the characteristic polynomial $\textrm{char}_A(x)$; 
        \item[(b)] every root of $p(x)$ is an eigenvalue of $A$; and
        \item[(c)] $\deg p \leq \operatorname{dim}(U)$.
    \end{enumerate}
    If, in addition, $U$ has the form $U = \textrm{span}\{\vec{v}, A \vec{v}, \ldots, A^{d-1} \vec{v}\}$ (or $U = \{0\}$ when $\operatorname{dim}(U)=0$) where $A^{j}\vec{v}$ are linearly independent for $j = 0, \ldots, d-1$, then 
    \begin{enumerate}
        \item[(d)] $\deg p = \operatorname{dim}(U)$.
        \item[(e)] Condition \eqref{def:monicpoly} is equivalent to $p(A) \vec{v} = 0$.
    \end{enumerate}
\end{theorem}
We include a proof of \Cref{Thm:AinvariantSpaces} in \Cref{app:proofs}, using straightforward generalizations of textbook arguments: $p(x)$ is the minimal polynomial for $A$ restricted to the subspace $U$ (from which properties (a--e) follow).  
\Cref{Thm:AinvariantSpaces} can then be extended to the case when $A$ has both a left and right orthogonal invariant subspace.
\begin{theorem}\label{Thm:OrthinvSpaces} (Left and right orthogonal $A$-invariant subspaces)
    Let $A \in \mathbb{R}^{s\times s}$ be a real matrix and $U \subseteq \mathbb{R}^s$ and $V \subseteq \mathbb{R}^s$ orthogonal subspaces, where $U$ is $A$-invariant and $V$ is $A^T$-invariant. Denote the minimal polynomials in \eqref{def:monicpoly} from \Cref{Thm:AinvariantSpaces} as $p(x)$ for $U$ and $q(x)$ for $V$, i.e., $p(A) \vec{u} = 0 \; \forall \vec{u} \in U$ and $q(A^T) \vec{v} = 0 \; \forall \vec{v} \in V$. Then the product $p(x) q(x)$ divides the characteristic polynomial $\textrm{char}_A(x)$.
\end{theorem}
Again, a proof of \Cref{Thm:OrthinvSpaces} is included in \Cref{app:proofs} for completeness. These results have a direct consequence for RK schemes with high WSO.
\begin{example}\label{Ex:BreakPThm} To highlight \Cref{Thm:AinvariantSpaces} and \Cref{Thm:OrthinvSpaces}, consider
\begin{equation*}
    A = \begin{pmatrix}
        \phantom{-}1 & \phantom{-}0 & 0 & 0 \\
        \phantom{-}1 & \phantom{-}1 & 0 & 0 \\
        \phantom{-}2 &           -1 & 1 & 0 \\
                  -2 & \phantom{-}0 & 1 & 2
    \end{pmatrix},\quad
    U = \begin{pmatrix}
        1 & 0 & \phantom{-}1 \\
        0 & 1 & -1 \\
        0 & 1 & \phantom{-}1 \\
        1 & 0 & -1
    \end{pmatrix}, \quad
    \vec{v} = \begin{pmatrix}
        \phantom{-}1 \\
        \phantom{-}1 \\ 
        -1 \\
        -1
    \end{pmatrix}.
\end{equation*}
The matrix $A$ has an $A$-invariant subspace spanned by the columns of $U$, and an orthogonal $A^T$-invariant space spanned by $\vec{v}$. The associated minimal polynomials are $p(x) = (x-1)^3$, s.t.~$p(A) U = 0$, and $q(x) = (x-2)$, s.t.~$\vecpower{v}{T} q(A) = 0$. The product $p(x)q(x) = (x-1)^3(x-2) = \textrm{char}_A(x)$ equals, and thus divides, the characteristic polynomial.  The example here will also be used to illustrate the lemmas in \S\ref{sec:DIRKscheme_Results}.
\end{example}

\begin{lemma}\label{Cor:char_pol_factorization}
	Let $(A, \vec{b})$ be a Runge-Kutta method. Then there exists a unique (non-zero, real coefficient) monic polynomial $Q(x)$ of minimal degree such that 
	\begin{equation}\label{Eq:Qdef}
	    \vecpower{b}{T} Q(A) = 0\;.
	\end{equation}
	The polynomial $Q(x)$ has $\deg Q = \dim Y$ and divides $\textrm{char}_{A}(x)$. 

	If, in addition, the RK method has WSO $q\ge 2$, then there exists a unique (non-zero, real coefficient) monic polynomial $P(x)$ of minimal degree, such that 
	\begin{equation} \label{Eq:Pdef}
	    P(A) \vec{\tau}^{(k)} = 0, \quad \mathrm{for} \;  k = 2, \ldots, q\;.
	\end{equation}
	Moreover, $\deg P \leq \dim K_q$, and the product $P(x)Q(x)$ divides $\textrm{char}_{A}(x)$.
\end{lemma}
In (\ref{Eq:Qdef}--\ref{Eq:Pdef}), $K_q$ and $Y$ are the subspaces defined by \eqref{Def:Kspace} and \eqref{Def:Yspace}; and $\vec{\tau}^{(k)}$ are the stage order residuals. We refer to $Q(x)$ and $P(x)$ in \Cref{Cor:char_pol_factorization} as the \emph{minimal polynomials} for $Y$ and $K_q$, respectively. 
\begin{proof}
    For the existence of $Q$ and its properties, set $U = Y$ and apply \Cref{Thm:AinvariantSpaces}(d--e). 
    For the existence of $P$ and its properties, set $U = K_q$ and $V=Y$ and apply \Cref{Thm:OrthinvSpaces}, where we denote the polynomials $p(x)$, $q(x)$ in \Cref{Thm:OrthinvSpaces} as $P(x)$, $Q(x)$ respectively. The only point to prove is that the condition
    \begin{equation} \label{Eq:FullPDef}
        P(A) \vec{v} = 0, \; \forall \vec{v} \in K_q\;,
    \end{equation}
    from \Cref{Thm:OrthinvSpaces} is equivalent to \eqref{Eq:Pdef} where $\vec{v}$ is restricted to be the stage order residuals $\vec{\tau}^{(k)}$ for $k = 1, \ldots, q$. Clearly, \eqref{Eq:FullPDef} implies \eqref{Eq:Pdef} since $K_q$ trivially includes $\vec{\tau}^{(k)}$ for $k = 1, \ldots, q$. Conversely, \eqref{Eq:Pdef} implies that every basis vector in $K_q$, i.e., $A^j \vec{\tau}^{(k)}$, satisfies $P(A) A^j \vec{\tau}^{(k)} = 0$. Hence, \eqref{Eq:Pdef} may be used in lieu of \eqref{Eq:FullPDef} to define the minimal polynomial $P(x)$.  
\end{proof}

Combined with \eqref{Dim:inequality} and \Cref{Thm:BoundOnDegR_ForDIRK}, \Cref{Cor:char_pol_factorization} implies the following:
\begin{corollary}\label{Cor:BoundonDimK}
    For an $s$-stage Runge-Kutta scheme with order $p \geq 1$ and weak stage order $q$ (with $K_q$ and $P(x)$ defined in \eqref{Def:Kspace} and \eqref{Eq:Pdef}), one has
    \begin{align*} 
        \deg(P) \leq \dim{(K_q)} \leq 
        \left\{ 
        \begin{array}{cl}
        s - \left\lfloor \tfrac{p+1 + \sigma}{2} \right\rfloor, & \textrm{for fully implicit schemes}, \\
        s - p - 1 + \sigma, & \textrm{for diagonally implicit schemes},
        \end{array}
        \right.        
    \end{align*}
    where $\sigma = 1$ if $A$ is invertible and the method is stiffly accurate, and $\sigma = 0$ otherwise.
\end{corollary}

\begin{remark}\label{Rmk:RootsPartitioning} (Eigenvalues of $A$)
A consequence of \Cref{Cor:char_pol_factorization} is that the eigenvalues of $A$ (including multiplicity) for a RK scheme can be partitioned into three sets: the roots of $P(x)$, the roots of $Q(x)$, and the roots of $N(x)$, where $\textrm{char}_A(x) = P(x) Q(x) N(x)$. \myremarkend
\end{remark}
\begin{remark}\label{Rmk:SDIRK_PQ}
    For an SDIRK with diagonal entries of $\gamma$ in $A$, the polynomials $P,Q$ have the form $Q(x) = (x - \gamma)^{d}$ ($d = \dim Y$), and $P(x) = (x- \gamma)^{\mu}$ ($\mu \leq \dim K_q$). \myremarkend
\end{remark}

Broadly speaking, the polynomials $Q(x)$ and $P(x)$ are algebraic objects associated with the corresponding geometric (orthogonal) spaces $K_q$ and $Y$. In the next sections, we use the polynomial $P(x)$ to obtain an expression for the stability function of RK schemes (particularly useful for schemes with WSO), and use $Q(x)$ to obtain necessary conditions for high WSO schemes (useful in the construction of WSO schemes).

\section{Stability Function in a Basis for \texorpdfstring{$Y$}{Y}}
\label{Subsec:stabilityfunction}
In general, the Runge-Kutta stability function has a restricted form when $p$ is close to (the optimal) $2s$, and also when $p \geq s$ (e.g., see  \cite[Chap.~IV.3, p.~47]{wanner1996solving}). In this section, we extend this constrained structure for $R(z)$ to the more general setting $p \geq \dim Y$ and to when $p$ is close to $2 \dim Y$ (which is a setting applicable for WSO). The key idea is to write $Q(x)$, and $Y$, in a basis of orthogonal polynomials related to the (\emph{tall-tree}) $p$-th order conditions
\begin{equation}\label{bnd_dimK_RK_Eq2}
	\vecpower{b}{T} A^j \vec{e} = \frac{1}{(j+1)!}, \quad 0\leq j \leq p-1\;.
\end{equation}
Specifically, define the linear functional $\mathcal{L}(\cdot)$ on the space of monomials $x^n$ with \emph{moments} $(\mu_n)_{n\geq 0}$ as
\begin{equation}\label{Def:LinFunctional}
    \mathcal{L}(x^n) := \mu_n, \; \textrm{where} \; \mu_n := \frac{1}{(n+1)!}, \; \textrm{for} \; n \geq 0\;.
\end{equation}
Associated to $\mathcal{L}$ are the Hankel moment matrices $\mathbf{H}_n = ( \mu_{i+j-2})_{i,j = 1}^{n}$.
\Cref{app:HankelMat} shows that $\det \mathbf{H}_n \neq 0$ 
  for all $n\geq 0$, which is a sufficient condition to construct a basis of polynomials $(Q_j)_{j \geq 0}$ satisfying the orthogonality relation
\begin{equation*} 
    \mathcal{L} \left( Q_i Q_j \right) = \zeta_i \delta_{ij}, \; \textrm{where} \; 
    \zeta_0 = 1, \; \textrm{and} \;
    \zeta_i = \frac{ \det(\boldsymbol{H}_{i+1})}{\det(\boldsymbol{H}_{i})}, \; \textrm{if} \; i \geq 1\;.
\end{equation*}
Here $\delta_{ij}$ is the Kronecker delta. For the moments in \eqref{Def:LinFunctional}, the first two orthogonal polynomials are $Q_0(x) = 1$ and $Q_1(x) = x - 1/2$; subsequent polynomials $Q_j(x)$ satisfy the three term recursion relation given by (see \Cref{app:HankelMat} for details):
\begin{equation}\label{Eq:TwoTermRecursion}
    Q_{n+1}(x) = x Q_n(x) + \xi_n^2 Q_{n-1}(x), \quad \xi_n^2 = \frac{1}{4(4n^2 -1)}, \; \textrm{for} \; n \geq 1\;.
\end{equation}
For instance, the next two polynomials are
\begin{equation*}
	Q_2(x)=x^{2}-\tfrac{1}{2} x+\tfrac{1}{12} \quad \textrm{and} \quad
	Q_3(x)=x^{3}-\tfrac{1}{2} x^{2}+\tfrac{1}{10} x-\tfrac{1}{120}\;. 
\end{equation*} 
Any degree $d$ polynomial can then be written in terms of the basis $\{Q_1, \dots, Q_d\}$. We use $\alpha_j$ to denote the coefficients of the minimal polynomial for $Y$ in this basis:
\begin{equation}\label{Eq:AltBasis}
    Q(x) = Q_d(x) + \alpha_{d-1} Q_{d-1}(x) +  \ldots
    + \alpha_0 Q_0(x), \quad \textrm{where} \; d := \dim Y\;.
\end{equation}
The order conditions then constrain \eqref{Eq:AltBasis} as follows.
\begin{lemma}\label{bnd_dimK_RK} ($Q$ is ``orthogonal'' to $Q_j$) 
    Consider a RK method with coefficients $(A,\vec{b})$, and let $Q$ be
    the minimal polynomial for the subspace $Y$, as defined in \eqref{Def:Yspace}, written in the form \eqref{Eq:AltBasis}. If the method's stability function is $R(z) = e^z + \BigO(z^{p+1})$ as $z\rightarrow 0$, then $\alpha_j = 0$ for $j \leq p - d - 1$.
\end{lemma}
\begin{proof}
Set $N := p - d - 1 \geq 0$ (if $N < 0$ there is nothing to prove). The $p$-th order conditions \eqref{bnd_dimK_RK_Eq2} imply  $\vecpower{b}{T} Q(A)A^\ell \vec{e} = \mathcal{L}( x^\ell Q(x))$ for all $0 \leq \ell \leq N$. On the other hand, $\vecpower{b}{T} Q(A) = 0$. Hence,  $\mathcal{L}( Q_\ell Q) = \zeta_{\ell} \alpha_{\ell} = 0$ for all $0 \leq \ell \leq N$, so $\alpha_j = 0$ for $j \leq N$ (since $\zeta_j \neq 0$). If $N \geq d$, this would imply $\mathcal{L}( Q Q_d)= \zeta_d = 0$, which is a contradiction. Note: $p \leq 2d$ or $\lfloor \frac{p+1}{2} \rfloor \le \dim(Y)$ when $p \geq 1$ and $d \geq 0$ are integers.
\end{proof}

Utilizing the the basis $(Q_j)_{j \geq 0 }$, we now work out an expression for the stability function $R(z)$.  We first expand $\vecpower{b}{T} (I - z A)^{-1}$ in the Krylov basis $\vecpower{b}{T}  Q_0(A)$,$\ldots$, $\vecpower{b}{T}  Q_{d-1}(A)$ with (to be determined) coefficients $\RzCoeff_j = \RzCoeff_j(z)$ functions of $z$:
\begin{equation}\label{Eq:RKBasis}
    \vecpower{b}{T} (I - z A)^{-1} = \vecpower{b}{T} \big( \RzCoeff_{d-1} Q_{d-1}(A) + \RzCoeff_{d-2} Q_{d-2}(A) + \ldots + \RzCoeff_0 Q_0(A) \big)\;.
\end{equation}
Right-multiplying \eqref{Eq:RKBasis} by $I - zA$ and using the fact that $\vecpower{b}{T} Q(A) = 0$, leads to the algebraic equation for the coefficients:
\begin{equation} \label{Eq:ModMu}
    (1 - zx) \big( \RzCoeff_{d-1} Q_{d-1}(x) + \RzCoeff_{d-2} Q_{d-2}(x) + \ldots + \RzCoeff_0 Q_0(x) \big) \equiv 1 \mod Q(x)\;.
\end{equation}
To solve for $\vec{\RzCoeff} = \left( \RzCoeff_0, \RzCoeff_1, \ldots, \RzCoeff_{d-1}\right)^T$, write $\vec{Q}(x) = \left( Q_0(x),  Q_1(x), \ldots, Q_{d-1}(x) \right)^T$. Using relation \eqref{Eq:TwoTermRecursion}, multiplication $x\vec{Q}$ can be written as a matrix multiplication $\mathbf{S}\vec{Q}$:
\begin{equation*}
    x \vec{Q} 
    = \!\begin{pmatrix}
        Q_1 + \dfrac{1}{2} Q_0 \\
        Q_2 - \xi_1^2 Q_0 \\
        \vdots \\
        Q_d - \xi_{d-1}^2 Q_{d-2} 
    \end{pmatrix}\!
    = \mathbf{S} \vec{Q}, \;
    \textrm{where} \; 
    \mathbf{S} := \!\begin{pmatrix}
         \frac{1}{2}  & 1     & 0 & \cdots   & 0\\ 
        \!-\xi_1^2          & 0     & 1 &          & \vdots \\
        0             & -\xi_2^2  & 0 &          &   \\
        \vdots        &       &   & \ddots   & 1 \\
        0             & \cdots&   & -\xi_{d-1}^2 & 0 
    \end{pmatrix}\!
    - \vec{e}_d \vec{\alpha}^T.
\end{equation*}
Here $\vec{e}_j$ is the $j$-th unit vector and $\vec{\alpha} = (\alpha_0, \ldots, \alpha_{d-1})^T$ are the coefficients of $Q(x)$ in \eqref{Eq:AltBasis}. Using these notations and relations, \eqref{Eq:ModMu} becomes:
\begin{align*}
    (1 - zx) \vec{\RzCoeff}^T \vec{Q} \equiv 1 \mod Q(x)\;, \\
    \left( \left( I - z \mathbf{S}^T \right) \vec{\RzCoeff}  - \vec{e}_1 \right)^{\!T\!} \vec{Q} \equiv 0 \mod Q(x)\;.
\end{align*}
The vector $\vec{\RzCoeff}$ then has solution:
\begin{equation}\label{Eq:ForMu}
    \left( I - z \mathbf{S}^T \right) \vec{\RzCoeff} = \vec{e}_1, \quad 
    \textrm{and via Cramer's rule:} 
    \quad
    \RzCoeff_{j-1}(z) = \frac{ \det\big( ( I - z \mathbf{S}^T )_j \big)  }{\det( I - z \mathbf{S} )}\;,
\end{equation}
where $(I - z \mathbf{S}^T)_j$ is the matrix $I - z \mathbf{S}^T$ with column $j$ replaced by $\vec{e}_1$. Substituting the expression for $\vecpower{b}{T}(I - z A)^{-1}$ into $R(z)$ yields:
\begin{equation}\label{Eq:StabFunBasisExp}
    R(z) = 1 + \sum_{j=0}^{d-1} \left( \vecpower{b}{T} Q_j(A) \vec{e} \right) z \RzCoeff_j(z)\;.
\end{equation}
Finally, the orthogonality property \Cref{bnd_dimK_RK} implies that $\vecpower{b}{T} Q_j(A) \vec{e}$ agrees with $\mathcal{L}(Q_j) = \delta_{0j}$ for all $j \leq p - 1$. This (significantly) simplifies the summation in \eqref{Eq:StabFunBasisExp}:
\begin{lemma}\label{Lem:R_Function_Alpha} (Stability function when $p \geq \dim Y$)
    If $R(z) = e^z + \BigO(z^{p+1})$ is a $p$-th order approximation as $z\rightarrow 0$ and $p \geq \dim Y$, then 
    \begin{equation}\label{Eq:RSimple}
        R(z) = 1 + z \RzCoeff_0(z)\;,
    \end{equation}
    where $\RzCoeff_0(z)$ is given by \eqref{Eq:ForMu} and depends only on the expansion of $Q(x)$ from \eqref{Eq:AltBasis}, i.e., $R(z)$ is a function of $\alpha_j$ (for $j = p - \dim Y, \ldots, \dim Y - 1$).  
\end{lemma}

To the authors' knowledge, the orthogonal polynomials $(Q_j)_{j\geq 0}$ have not 
previously been introduced in the context of RK schemes. However, they are
connected with existing ideas in the (extensive) literature on RK methods. For instance, the values $\xi_j$ defined in \eqref{Eq:TwoTermRecursion} appear in the $W$-transform and in recurrence relations for the shifted Legendre polynomials \cite[Chap.~IV.5]{wanner1996solving} commonly used in the construction of high stage order schemes. \Cref{bnd_dimK_RK} is also similar in spirit to the orthogonality relation Lemma~5.15 in \cite[Chap.~IV.5]{wanner1996solving} which expands the characteristic polynomial of $C$ (when $\numc = s$) in the shifted Legendre polynomial basis. 

\begin{remark}\label{rmk:OrthPolynomials} (Orthogonal polynomials with respect to a linear functional)
In the (semi)classical theory of orthogonal polynomials, the Hankel (moment) matrices are positive definite and $\mathcal{L}(\cdot)$ defines an inner product as an integral with respect to a (non-negative) measure \cite{RegOrthPoly, Ismail2005}. In contrast, for the matrices $\mathbf{H}_d$ considered here, both $\det \mathbf{H}_d \neq 0$ and $\zeta_i \neq 0$.  However, $\zeta_i < 0$ for some values of $i$, so $\mathcal{L}(\cdot)$ does not have an inner product representation. While much of the (semi)classical theory of orthogonal polynomials still holds \cite{RegOrthPoly} (e.g., the three-term recurrence), some properties do not. For instance, the basis polynomials $Q_j(x)$ have complex roots, whereas in the classical theory of orthogonal polynomials the roots are real and simple. \myremarkend
\end{remark}

\begin{remark}\label{rmk:MomGenFun} (Moment-generating function of $\mathcal{L}$) 
The moments $(\mu_n)_{n\geq 0}$ do not arise by accident; their generator function is $G(z) = (e^{z} - 1)/ z = \sum_{n \geq 0} \mu_n z^n$, whose Taylor expansion agrees up to a suitable order with $(R(z) - 1)/z$. The Hankel determinants, or coefficients of the recurrence relation \eqref{Eq:TwoTermRecursion} appear in the J-fraction (continued fraction) of the generating function.  For $G(z)$, this continued fraction appears in the RK literature through Pad\'{e} approximations of $e^z$ (see e.g.\ Exercise~4 in Chap.~IV.3 or Theorem~5.18 of \cite{wanner1996solving}). \myremarkend
\end{remark}

\begin{remark}\label{rmk:stiffaccuracy} (Stiff accuracy)
Equation \eqref{Eq:RSimple} can be extended to incorporate the structure imposed by stiff accuracy (for invertible $A$) using an alternative polynomial basis $(\tilde{Q})_{n\geq 0}$. Specifically, let $\tilde{\mu}_0 = 1$, $\tilde{\mu}_{n+1} := \mu_{n}$ for $n \geq 0$ and $\tilde{\mathcal{L}}(x^n) := \tilde{\mu}_n$. The Hankel matrices with moments $\tilde{\mu}_n$ are shown in \eqref{Eq:GenHankel} to have non-zero determinants and define a basis orthogonal with respect to $\tilde{\mathcal{L}}$, i.e., $\tilde{Q}_0(x) = 1$, $\tilde{Q}_1(x) = x - 1$, etc.  The stability function (having one less degree of freedom), for stiffly accurate schemes  can then be obtained in this basis. \myremarkend

\end{remark}

\section{Necessary Conditions on \texorpdfstring{$P(x)$}{P(x)} for DIRK Schemes}
\label{sec:DIRKscheme_Results}
\Cref{Thm:MainResultDIRK} provides a bound on the WSO for a method with order $p$ and number of stages $s$. However, the theorem does not explain how one might go about constructing schemes with high WSO. In this section we examine the solvability of the equations $P(A) \vec{\tau}^{(k)} = 0$ for the matrices $\lbmap A \rbmap_j$. The results impose necessary conditions on $P(x)$ (e.g., constraints on the roots) for the construction of high WSO schemes. The necessary conditions also restricts how the spectrum of $A$ is partitioned into the minimal polynomials $Q(x)$, $P(x)$. We focus in this section on schemes that are not GEDIRK schemes.  

Here we write $p_{j}(x)$ to denote the minimal polynomial of $\lbmap A \rbmap_{j}$ when $j \geq 1$ (we have $p_{j}(x) = 1$ if $j = 0$).  The main result is that $p_r$ must divide $P(x)$ when the first $r$ abscissas are distinct.
\begin{lemma}\label{Lem:PropKqDIRK_NonConfluent} (Necessary condition on $P(x)$ for a DIRK)
    Let a DIRK scheme that is not a GEDIRK be given with WSO $q \geq 2$. If the abscissas $\{c_1, c_2, \ldots, c_{r}\}$ are distinct for $r \leq \lfloor \frac{q}{2} \rfloor$ then $p_{r}(x)$ divides $P(x)$.
\end{lemma}
\begin{proof}
    Apply the map $\lbmap \cdot \rbmap_r$ to \eqref{Eq:Pdef} and use \eqref{Eq:CommutationProperty} to obtain:
    \begin{equation}\label{Eq:DIRK_SubMatrixWSO}
        P( \lbmap A \rbmap_r ) \; \tau^{(k)}( \lbmap A \rbmap_r) = 0, 
        \; \textrm{for} \; k = 2, \ldots, q\;.
    \end{equation}
    By \Cref{lem:Refinementcj_zero} we have that $\lbmap K_q \rbmap_r = \mathbb{R}^r$, which when combined with \eqref{Eq:DIRK_SubMatrixWSO} implies $P( \lbmap A \rbmap_r ) = 0$. Thus $p_r(x)$ divides $P(x)$. 
\end{proof}
\begin{corollary} For any non-confluent DIRK scheme (that is not a GEDIRK scheme), $p_{r}(x)$ divides $P(x)$ where $r=\lfloor \frac{q}{2} \rfloor$. 
\end{corollary}

\Cref{Ex:BreakPThm} demonstrates that the assumption
of distinct abscissas in \Cref{Lem:PropKqDIRK_NonConfluent} 
is necessary.  It provides a confluent scheme ($\vec{c} = (1,2,2,1)^T$) that is $S$-irreducible and for which the minimal polynomial of $A$ (i.e., $p_4(x) = (x-1)^3 (x-2)$) is equal to the characteristic polynomial and does not divide $P(x) = (x-1)^3$.

One might wonder (I)~whether the assumption of distinct abscissas in \Cref{Lem:PropKqDIRK_NonConfluent} can be relaxed; and (II)~whether $p_j(x)$ must ever be the characteristic polynomial of $\lbmap A \rbmap_j$. In general the answer to both questions is no (see \Cref{Ex:BreakPThm}). However, due to the small number of variables in the top block of $\lbmap A \rbmap_j$ the answer to both questions is yes when $j \leq 3$ in (I) and $j \leq 2$ in (II), provided that the scheme is $S$-irreducible.

\begin{definition} (S-reducible, Def.~12.17 in \cite{wanner1996solving}) A RK method is S-reducible, if for some partition $(S_1, \ldots, S_{\Sredr})$ of $\{1, \ldots, s\}$ $(\Sredr < s)$ the indicator vectors\footnote{In a slight abuse of notation we use $\vec{S}_j$ to denote both the set and the indicator vector. The notation $\vec{S}_j$ is non-standard---the definition typically writes \eqref{Eq:SReducible} as a summation.} $\vec{S}^{(m)} = \sum_{j \in S_{m}} \vec{e}_j$ satisfy for all $\ell$, $m$:
\begin{equation}\label{Eq:SReducible}
    (\vec{e}_i - \vec{e}_j)^T A \vec{S}^{(m)} = 0, \quad 
    \textrm{if} \; i,j \in S_\ell\;.
\end{equation}
\end{definition}
An S-reducible scheme is equivalent to a smaller $\Sredr$-stage RK scheme $(A^*, \vec{b}^*)$ where each partition $S_j$ is replaced by a single stage (see Eq.~(12.24) of \cite[Chapt.~IV.12]{wanner1996solving}; the smaller scheme generates the same stage solutions $g_i$ in \eqref{eq:RK_methods_1}). For S-reducible DIRKs, the new scheme can also be made a DIRK\footnote{Formally, if we order the partitions so that $\min \{ x \in S_i \} < \min \{y \in S_j\}$ whenever $i < j$. Then for $1 \leq i \leq \Sredr$ define the new scheme (which is a DIRK) as $a_{ij}^* = \sum_{k \in S_j} a_{ij}$, $b_j^* = \sum_{k \in S_j} b_j$.}.

\begin{remark}\label{Rmk:EqualTimeRed} (A simple $S$-reducibility observation) 
    Any DIRK scheme with $c_1 = c_2$ is $S$-reducible: apply Def.~12.17 in \cite{wanner1996solving}, where the partition of equivalent stages (i.e., partition of the integers $\{1, \ldots, s\}$) is taken as $S_{1} = \{1, 2\}$ and $S_2 = \{3\}$, $\ldots$, $S_{s - 1} = \{s\}$. Then \cite[Thm.~2.2]{hundsdorfer1980note} implies that the first $2$ stages of $A$ yield the same intermediate stage value solutions---and can be replaced by a single stage.
\end{remark}

\begin{lemma}\label{lem:CharPolyBarriers} (A general necessary condition on $P(x)$)
    Let an  $S$-irreducible DIRK scheme (that is not a GEDIRK) be given with coefficient matrix $A$ and WSO $q$, such that $\lbmap A \rbmap_3$ is invertible. 
    \begin{enumerate}
        \item[(a)] If $q > 1$, then $\deg P \geq 1$ and $a_{11}$ is a root of $P$.
        \item[(b)] If $q > 3$, then $\deg P \geq 2$, and $a_{11}, a_{22}$ are roots of $P$.
        \item[(c)] If $q > 5$, then $P(x) = p_3(x) \tilde{P}(x)$ where $p_3(x)$ is the minimal polynomial of $\lbmap A \rbmap_3$ and is divisible by $(x-a_{11}) (x-a_{22})$.
    \end{enumerate}
\end{lemma}
\begin{proof}
(a)~Since the set $\{ c_1 \}$ contains 1 distinct element (trivially), we can apply \Cref{Lem:PropKqDIRK_NonConfluent} with $r = 1$ when $q > 1$, so that $p_1(x) = (x - a_{11})$ divides $P(x)$.

For (b), \Cref{Rmk:EqualTimeRed} implies that $c_1 \ne c_2$ for any $S$-irreducible method. Applying \Cref{Lem:PropKqDIRK_NonConfluent} with $r = 2$ when $q > 3$ implies that $p_2(x)$ divides $P(x)$. If $p_2(x) = (x - a_{11})$ has only $1$ root (which must be $a_{11}$ by (a)), then $p_2( \lbmap A \rbmap_2) = 0$ implies $\lbmap A \rbmap_2 = a_{11} I$, whence $c_1 = c_2 = a_{11}$ (which is a contradiction). Thus, $p_2(x) = (x - a_{11})(x-a_{22})$ is the characteristic polynomial of $\lbmap A \rbmap_2$. 

For (c), suppose by way of contradiction that $p_3(x)$ does not divide $P(x)$, so by (b), $p_2(x)$ is the minimal polynomial of $\lbmap A \rbmap_3$ restricted to $\lbmap K_q \rbmap_3$. When $q > 5$, we claim that 
\begin{align}\label{Eq:Kq_For3}
    \lbmap K_q \rbmap_3 = \textrm{span}\{ \lbmap\vec{e} \rbmap_3, \lbmap \vec{c}\rbmap_3\}, 
\end{align}
where either $c_1 = c_3 \neq c_2$ or $c_2 = c_3 \neq c_1$. The reasons are: 
(i)~\Cref{Lem:PropKqDIRK_NonConfluent} implies that $\vec{c}$ is confluent (with $c_1 \neq c_2$) so $\numc = 2$ and (by the same argument in the proof) $\dim \lbmap K_q\rbmap_3 \leq 2$;
(ii) if $c_j = 0$, \eqref{Eq:Kq_For3} follows immediately from \Cref{thm:GenPropertiesKq}(b); 
(iii) If $c_j = 0$, the construction in \Cref{lem:Refinementcj_zero} implies $\lbmap K_q \rbmap_3 = \textrm{span}\{ \lbmap\vec{c} \rbmap_3, \lbmap A \vec{c}\rbmap_3\}$ which is \eqref{Eq:Kq_For3} when $\lbmap A \rbmap_3$ is invertible (i.e., multiply by $\lbmap A \rbmap_3^{-1}$). For $\lbmap K_q \rbmap_3$ to be $A$-invariant requires 
\begin{equation}\label{Eq:IrrAExamples}
    A = \!\begin{pmatrix}
        a_{11} & 0 & 0 \\
        a_{21} & a_{22} & 0 \\
        a_{11} \!-\! a_{33}\! & 0  & a_{33}
    \end{pmatrix}
    \textrm{if} \; c_1 \!=\! c_3, \;
    A = \!\begin{pmatrix}
        a_{11} & 0 & 0 \\
        a_{21} & a_{22} & 0 \\
        a_{21} & \!a_{22} \!-\! a_{33}\! & a_{33}
    \end{pmatrix}
    \textrm{if} \; c_2 \!=\! c_3\;.
\end{equation}
In both cases $A$ is $S$-reducible with partitions $S_1 = \{1, 3\}$, $S_{2} = \{ 2\}$ ($c_1 = c_3$) or $S_1 = \{1\}$, $S_{2} = \{ 2, 3\}$ ($c_2 = c_3$). Hence $p_2(x)$ divides $p_3(x)$ which divides $P(x)$.
\end{proof}

While this section has focused entirely on $P(x)$, the final remark concerns a similar-in-spirit result for $Q(x)$.
\begin{remark} ($a_{ss}$ is a root of $Q(x)$)
A DJ-irreducible DIRK scheme has $\vecpower{e}{T}_s \vec{b} \neq 0$; otherwise, the scheme is independent of the last stage. Since $0 = \vecpower{b}{T} Q(A) \vec{e}_s = Q(a_{ss}) \vecpower{b}{T} \vec{e}_s$, the entry $a_{ss}$ must be a root of $Q$, i.e., $Q(x) = (x-a_{ss})\tilde{Q}(x)$. 
\myremarkend 
\end{remark}

\section{Examples}
\label{sec:Examples}
This section provides some concrete examples that highlight how the theory established above can be used to reduce the number of degrees of freedom when constructing RK schemes with high WSO. In \Cref{Ex:DIRK1} and \Cref{Ex:DIRK2} we parameterize DIRK schemes with invertible $A$ and WSO $q = 3$. Based on \Cref{Rmk:RootsPartitioning} and \Cref{Lem:PropKqDIRK_NonConfluent}, we can determine that certain diagonal entries of $A$ (those closest to the top left) are roots of $P(x)$, while those near the bottom right are roots of $Q(x)$.

\begin{example}\label{Ex:DIRK1}
The theory enables a complete characterization of DIRK schemes with $(s,p,q) = (2,2,3)$. \Cref{Thm:BoundOnDegR_ForDIRK} and \Cref{Lem:PropKqDIRK} require: $\dim Y = 1$, $\dim K_q = 1$. Via \Cref{Lem:PropKqDIRK_NonConfluent} the associated polynomials are $P(x) = x- a_{11}$ and $Q(x) = x - a_{22}$. For a $p = 2$ method, the orthogonality property from \Cref{bnd_dimK_RK} implies that $Q(x) = Q_1(x) = x - \tfrac{1}{2}$, so $a_{22} = \tfrac{1}{2}$. The first row of equation $P(A) \vec{\tau}^{(k)} = 0$ for $k = 2,3$ is automatically satisfied; the second row yields two equations, the solution of which determines $A$. The last order condition $\vecpower{b}{T} \vec{e} = 1$ determines $\vec{b}$, resulting in two schemes:
\begin{equation}\label{Eq:wso3_p2_ex}
    A = \begin{pmatrix}
        1 \mp \tfrac{ \sqrt{2}}{2} & 0 \\
        \tfrac{1}{2} \pm \tfrac{\sqrt{2}}{2} & \tfrac{1}{2}
    \end{pmatrix}, 
    \quad 
    \vec{b} = 
    \begin{pmatrix}
    \tfrac{1}{2} \pm \tfrac{\sqrt{2}}{4} \\
    \tfrac{1}{2} \mp \tfrac{\sqrt{2}}{4} 
    \end{pmatrix}, \quad R(z) = \frac{1 + \tfrac{z}{2}}{1 - \tfrac{z}{2} }\;.
\end{equation}
Either choice of signs leads to a $(2,2,3)$ scheme. In both cases, $\vecpower{b}{T} A = \tfrac{1}{2}\vecpower{b}{T}$ is a left eigenvector of $A$ and $K_3 = \textrm{span}\{ 2\vec{c} - \vec{e}\, \} = W_3$. \end{example}

\begin{example}\label{Ex:DIRK2} (DIRK schemes with $(s, p, q) = (3,3,3)$)
The chosen values of $(s,p,q)$ satisfy \Cref{Thm:MainResultDIRK} sharply and require that
$\dim Y = 2$ and $\deg P = \dim K_q = 1$ with $P(x) = x - a_{11}$, $Q(x) = (x - a_{22}) (x-a_{33})$. Since $\lbmap A \rbmap_2$ satisfies \eqref{Eq:DIRK_SubMatrixWSO} it must be a scalar multiple (call the ratio $a$) of the Butcher matrix in \eqref{Eq:wso3_p2_ex}. The polynomial $Q(x) = Q_2(x) + \alpha_1 Q_1(x)$ then has $a_{22} = a/2$  as a root so that $\alpha_1 = -Q_2(\tfrac{a}{2}) / Q_1(\tfrac{a}{2})$, which uniquely determines (i)~$a_{33}$ as the second root of $Q(x)$; and (ii)~the stability function in terms of $a$. The matrix $A$ and $R(z)$ are constrained to be:
\begin{equation*}
    A = \begin{pmatrix}
        (1 \mp \tfrac{ \sqrt{2}}{2})a & 0 & 0 \\
        (\tfrac{1}{2} \pm \tfrac{\sqrt{2}}{2})a & \tfrac{1}{2}a & 0 \\
        a_{31} & a_{32} & \tfrac{3a - 2}{6(a-1)}
    \end{pmatrix},
    \quad
    R(z) = \frac{ (1+6 \alpha_1)z^2 + (6+12\alpha_1 )z + 12}{(1-6\alpha_1)z^2 + (12\alpha_1-6)z + 12}\;.
\end{equation*}
The values $(a_{31}, a_{32})$ can be parameterized in terms of $a$ via the $3$rd row equations of $P(A) \vec{\tau}^{(j)} = 0$ for $j = 2,3$. For each choice of $\pm$ in $A$, we have the following. There are $3$ solution branches for $(a_{31}, a_{32})$, two of which yield S-reducible schemes with the structures given in \eqref{Eq:IrrAExamples}. The one irreducible solution branch has that $\vec{\tau}^{(2)}$ and $\vec{\tau}^{(3)}$ are parallel. The vector $\vec{b}$ is determined via $\vecpower{b}{T}(\vec{e}, \vec{c}, \vec{\tau}^{(2)}) = (1,\tfrac{1}{2}, 0)$. By construction, $\vecpower{b}{T} Q(A) = 0$ so that $\vecpower{b}{T} \vec{e} = 1$ and $\vecpower{b}{T} \vec{c} = \tfrac{1}{2}$ ensure $\vecpower{b}{T} A^2 \vec{e} = \tfrac{1}{6}$. Finally, combining $\vecpower{b}{T} (A \vec{\tau}^{(1)} + \vec{\tau}^{(2)}) = 0$ with $\vecpower{b}{T} A^2 \vec{e} = \tfrac{1}{6}$ yields $\vecpower{b}{T} \vecpower{c}{2} = \tfrac{1}{3}$.

As expected, with $s = 3$ there is no choice of $a$ that satisfies the additional conditions imposed by order $p = 4$ or $q = 4$. Finally, $0 < a < \tfrac{2}{3}$ or $a > 1$ is required for $A$ to have positive eigenvalues and $R(z)$ to have all its poles in the right half plane.
\end{example}

\begin{example}
The bound in \Cref{Thm:MainResult1} can be sharp for both stiffly accurate schemes and EDIRKs. Specifically, stiffly accurate DIRK schemes with $(s,p,q) = (4,3,3)$ were constructed in an ad-hoc fashion in \cite{KetchesonSeiboldShirokoffZhou2020}, where a-posteriorily $s$ is observed to be sharp. Setting $\dim K_q = 0$ in \Cref{Lem:PropKqDIRK}, the fact $\qso \leq q$ recovers the known result that stage order is limited to $\qso = 2$ for EDIRKs and $\qso = 1$ for DIRKs. 
\end{example}

\begin{example}\label{Rmk:MainResultIsSharp} (Schemes with high stage order)
    The Gauss-Legendre RK methods satisfy the bound in \Cref{Thm:MainResult1} sharply since $p = 2s$, $q=p$ (so that $\dim(K_q)=0$) and $\numc = s$ (and are not stiffly accurate).  As an example of \Cref{Cor:BoundonDimK}, we obtain the bound $p \le 2s$, well known already from the theory of numerical quadrature.  
\end{example}

\begin{remark} (Guide to constructing DIRK schemes with high WSO)
For schemes with $q = p$, one can take $r = \lfloor q/2 \rfloor$, set $P(x) = (x-a_{11}) \cdots (x - a_{rr})$ and then solve \eqref{Eq:Pdef} as a sufficient condition (and use \Cref{Thm:MainResultDIRK} to guide the number of stages). \myremarkend
\end{remark}

\section{Conclusions and Outlook}
\label{sec:DIRK_conclusions}
Weak stage order can be viewed as a geometric condition that, when satisfied, can remove order reduction in Runge-Kutta schemes applied to linear problems with time-independent operators. The general theory of WSO provided here relates geometric objects (WSO invariant subspaces) to associated algebraic objects (minimal polynomials). This relationship allows us to establish order barriers for WSO that generalize the well-known bounds on the RK order $p$ in terms of $s$ (both for general RK schemes and DIRKs). Along the way, we also provide new formulas for the RK stability function in terms of a family of polynomials which are ``orthogonal'' with respect to a linear functional. The new necessary conditions show how one needs to split the spectrum of $A$ into the roots of $P(x)$ and $Q(x)$---which is of practical use when constructing high WSO schemes (i.e., beyond 3). Indeed, the necessary conditions were used in the companion work \cite{BiswasKetchesonSeiboldShirokoff2023} in the construction of new DIRK schemes with WSO 4 and 5 (and satisfy the full set of order conditions). 

Since SDIRK methods are a subset of DIRK methods, and ERK methods are a subset of GEDIRK
methods, the bounds in \Cref{Thm:MainResultDIRK} apply to these classes as well. 
It is natural to ask whether stricter bounds can be found for these smaller classes.
Based on \Cref{Thm:RealPoles}, it seems that no further improvement can be obtained
in the bounds on $\dim(Y)$ for SDIRK methods compared to DIRK methods, but it might be
possible to obtain tighter bounds by further exploiting the structure of $K_q$.  
Stricter bounds on ERK methods are an area of current research.

The results presented here give rise to several research directions. First, is \Cref{Thm:MainResultDIRK} sharp for all $p$, $q$, $s$? And, is it further constrained by the non-tall tree order conditions?  There is also the related (practical) issue of constructing DIRK schemes with $q = p$ (or $q = p-1$) and the fewest stages $s \sim \frac{2}{3} p$ satisfying \Cref{Thm:MainResultDIRK} (or better yet, analytically parameterizing such schemes). While the main results in this work do apply to all (E)DIRKs, DIRKs with additional constraints such as SDIRKs or ERKs are expected to further impact the bounds in \Cref{Thm:MainResultDIRK}.

\section*{Acknowledgments}
The authors gratefully acknowledge the feedback from two anonymous reviewers whose suggestions led to improvements in the paper.
This material is based upon work supported by the National Science Foundation under Grants No.\ DMS--2012271 (Biswas, Seibold), No.\ DMS--1952878 (Seibold), and No.\ DMS--2012268 (Shirokoff).
D. Shirokoff gratefully acknowledges the hospitality of Robert Pego and the Dept.\ of Mathematical Sciences at Carnegie Melon University where portions of this work were completed.

\appendix

\section{Proofs of Theorems~\ref{Thm:AinvariantSpaces} and \ref{Thm:OrthinvSpaces}}
\label{app:proofs}
We collect the linear algebra proofs here, which are adapted from various materials (e.g., \cite[Chap.~2.6]{Petersen2012}).
\begin{proof} (\Cref{Thm:AinvariantSpaces}, $A$-invariant subspaces)
Let $\mathbf{U} \in \mathbb{R}^{s \times d}$ form an orthogonal basis for $U$. Since $U$ is $A$-invariant, there is a square matrix $A_{UU} := \mathbf{U}^T A \mathbf{U} \in \mathbb{R}^{d\times d}$ such that (each column of $\mathbf{U}$ is mapped back into the column space of $\mathbf{U}$)
\begin{equation}\label{Eq:A_inv_space}
    A \mathbf{U} = \mathbf{U} A_{UU}, 
\end{equation}
For any polynomial $p(x)$, \eqref{Eq:A_inv_space} then implies that $p(A) \mathbf{U} = \mathbf{U} p(A_{UU})$. Thus, $p(x)$ in \eqref{def:monicpoly} is the minimal polynomial of $A_{UU}$ (which has real coefficients since $A_{UU}$ is real) because it satisfies $p(A) \mathbf{U} = 0$ and any other smaller degree polynomial fails to satisfy $p(A) \mathbf{U} = \mathbf{U} p(A_{UU}) \neq 0$.

For (a), $p(x)$ divides the characteristic polynomial of $A_{UU}$, which divides $\textrm{char}_A(x)$; (c) also follows since $\textrm{deg} \; p \leq \textrm{deg} \; \textrm{char}_{A_{UU}}(x) = \textrm{dim}(U)$. 

For (b), every root $\lambda$ of $p(x)$ is an eigenvalue of $A_{UU}$ and hence of $A$. 
    
For (d), set $\ell = \textrm{deg} \; p$. The condition $p(A) \vec{v} = 0$ (where $\vec{v}$ is the vector defining $U$) implies that $A^{\ell}\vec{v}$ is a linear combination of $\vec{v}, \ldots, A^{\ell - 1}\vec{v}$, which is only possible if $\textrm{dim}(U) \leq \textrm{deg} \; p$. Combined with part (c), this shows $\dim U = \deg p$. 
 
For (e), when $U = \{ \vec{v}, A\vec{v}, \ldots, A^{d-1} \vec{v}\}$, condition \eqref{def:monicpoly} implies that $p(A) \vec{v} = 0$ since $\vec{v} \in U$. Conversely, any element $\vec{u} \in U$ is a linear combination of $A^{j} \vec{v}$ for $j = 0, \ldots, d-1$. But then, $p(A) \vec{v} = 0$ implies (since $A^j$ and $p(A)$ commute) that $p(A) A^j \vec{v} = 0$ for any $j$, and hence $p(A) \vec{u} = 0$ for any $\vec{u} \in U$. 
\end{proof}

\begin{proof} (\Cref{Thm:OrthinvSpaces}, Left and right orthogonal $A$-invariant subspaces).
    Similar to the proof of \Cref{Thm:AinvariantSpaces}, introduce the orthogonal matrix $\mathbf{O} = ( \mathbf{U} \ | \ \mathbf{W} \ | \ \mathbf{V}) \in \mathbb{R}^{s\times s}$, where the columns of $\mathbf{U} \in \mathbb{R}^{s \times d_u}$
    and $\mathbf{V} \in \mathbb{R}^{s \times d_v}$ 
    form an orthogonal basis for $U$ and $V$, respectively (the columns of $\mathbf{W}$ span the remaining space). Now $A$ block-diagonalizes as
    \begin{equation*} 
        \mathbf{O}^T A \mathbf{O} = 
        \begin{pmatrix}
            A_{UU} & A_{UW} & A_{UV}  \\
             0 & A_{WW} & A_{WV} \\
             0 & 0 & A_{VV}
        \end{pmatrix},
    \end{equation*}
    where $A_{\Sigma \Theta} = \Sigma^T A \Theta$ where $\Sigma, \Theta \in \{ \mathbf{U}, \mathbf{V}, \mathbf{W}\}$. Hence,
    \begin{equation}\label{Eq:charpolyfactor}
        \textrm{char}_A(x)  = 
    \textrm{char}_{A_{UU}}(x) \; \textrm{char}_{A_{WW}}(x) \; \textrm{char}_{A_{VV}}(x)\;.
    \end{equation}
    Via \Cref{Thm:AinvariantSpaces} the polynomial $p(x)$ is the minimal polynomial of $A_{UU}$.
    \Cref{Thm:AinvariantSpaces} applies to $A^T$ with space $\mathbf{V}$, so that $q(x)$ is the minimal polynomial of $A_{VV}^T$ (and is the same as the minimal polynomial of $A_{VV}$). Hence, $p(x)q(x)$ divides $\textrm{char}_{A_{WW}}(x) \textrm{char}_{A_{VV}}(x)$, which, from \eqref{Eq:charpolyfactor}, divides $\textrm{char}_{A}(x)$. 
\end{proof}
It is generally not true that $p(x)q(x)$ divides the minimal polynomial of $A$.

\section{Proofs of Hankel Matrix Determinant and Orthogonal Polynomial Recurrence Relation}
\label{app:HankelMat}
Here we provide details for the determinant computation ($m \in \mathbb{Z}_{\geq 0}$) of
\begin{equation}\label{Eq:GenHankel}
    \mathbf{H}_n(m) = \begin{pmatrix}
        \tfrac{1}{m!} & \tfrac{1}{(m+1)!} & \cdots & \tfrac{1}{(m+n-1)!} \\
        \tfrac{1}{(m+1)!} & \ddots & & \tfrac{1}{(m+n)!} \\
        \vdots & & \ddots &  \vdots \\
        \tfrac{1}{(m+n-1)!} & \cdots & \cdots & \tfrac{1}{(m+2n - 2)!}
    \end{pmatrix} \in \mathbb{R}^{n\times n}.
\end{equation}
Let $M_n(m) := \det( \mathbf{H}_n(m))$. A formula (without proof) for $M_n(m)$ is given in \cite{DetEval}:
\begin{align}\label{Eq:DetFormula}
    M_n(m) &= \sigma(n) \frac{c(n) c(m+n-1)}{c(m+2n - 1)}, \quad \textrm{where} \;  
    c(n) := \Pi_{i=1}^{n-1} i!\;, \\
    \nonumber
    \textrm{and }
    \sigma(n) &:= 
    \left\{\begin{array}{cr}
     	\phantom{-}1, &\textrm{if } r = 0 \textrm{ or } 1, \\
     	-1, &\textrm{if } r = 2 \textrm{ or } 3,
   \end{array}\right. \textrm{ where } n \equiv r \mod 4\;.
\end{align}
Since we could not find a proof of \eqref{Eq:DetFormula} in the literature, we provide a brief one here. Following the approach in \cite[Method~2 in \S4]{Krattenthaler2005}, for any square matrix $A$, a determinant formula due to Jacobi reads:
\begin{align}\label{Eq:JacIdentity}
    \det A \cdot \det A_{1,n}^{1,n} = \det A_1^1 \cdot \det A_{n}^n - \det A_{1}^{n} \cdot \det A_{n}^1\;,
\end{align}
where $A_{i}^j$ denotes the submatrix of $A$ with row $i$ and column $j$ removed. Applying \eqref{Eq:JacIdentity} to the matrix $M_n(m)$ and using the symmetry of the Hankel matrix yields the recursion:
\begin{equation}\label{Eq:IterIdentity}
    M_n(m) M_{n-2}(m+2) = M_{n-1}(m+2) M_{n-1}(m) - M_{n-1}(m+1)^2\;.
\end{equation}
Formula \eqref{Eq:IterIdentity} relates $M_n(m)$ to matrices of size $n-1$ and $n-2$, and can be used to prove \eqref{Eq:DetFormula} via induction. Note that for $M_1(m) = 1/m!$ and $M_2(m) = -[(m+1)! (m+2)!]^{-1}$, formula \eqref{Eq:DetFormula} is readily verified (and holds $\forall m \geq 0$).  Assume \eqref{Eq:DetFormula} holds for $1 \leq k < n$. We can then divide \eqref{Eq:IterIdentity} through by $M_{n-1}(m+1) M_{n-2}(m+2)$ (which are non-zero by assumption) to obtain:
\begin{equation}\label{Eq:IterIdentity2}
    \tfrac{M_n(m)}{M_{n-1}(m+1)} = \tfrac{M_{n-1}(m+2) M_{n-1}(m)}{M_{n-2}(m+2) M_{n-1}(m+1)} - \tfrac{M_{n-1}(m+1)}{M_{n-2}(m+2)}\;.
\end{equation}
It is then a matter of substituting the formulas from \eqref{Eq:DetFormula} into \eqref{Eq:IterIdentity2} to verify (after several lines of factorial cancellations) that the following holds\footnote{Note that $\sigma(n)/\sigma(n-1)$ always has opposite sign to $\sigma(n-1)/\sigma(n-2)$.}:
\begin{equation*} 
    \tfrac{(n-1)!}{(m+2n-2)!} = -\left(
    \tfrac{(n-2)! (m+n-1)! (m+2n-3)!}{(m+2n-2)! (m+2n-3)! (m+n-2)!} - \tfrac{(n-2)!}{(m+2n-3)!}\right)\;.
\end{equation*}

Since (for every $n$) the Hankel matrix determinants $M_n(m)$ do not vanish, the entries of $\mathbf{H}_n(m) $ define a \emph{quasi-definite linear functional} which then have associated orthogonal polynomials. Monic orthogonal polynomials (always) satisfy a three term linear recurrence of the general form
\begin{equation}\label{Eq:TwoTermRec}
    Q_{n+1}(x) = (x + \gamma_n ) Q_{n}(x) + \beta_n Q_{n-1}(x)\;,
\end{equation}
where $\mathbf{H}_n(m)$ is initialized via $Q_0(x) = 1$, $Q_1(x) = x - 1/(m+1)$. We now obtain formulas for $\beta_n$, $\gamma_n$ via determinant computations. From \cite[Thm.~29]{Krattenthaler2005} we have:
\begin{align*}
    &\beta_n(m) = -\tfrac{M_{n-1}(m)}{M_{n}(m)}\cdot \tfrac{M_{n+1}(m)}{M_n(m)} 
    = \tfrac{n (m + n - 1)}{(m+2n)(m+2n-1)^2(m+2n-2)}\;, \\
    &\textrm{When } m = 1: \quad 
    \beta_n = \tfrac{1}{4(4n^2 -1)}\;.
\end{align*}
The $\gamma_n$ depend only on the second leading term $\lambda_n(m)$ of $Q_n(x) = x^n + \lambda_n(m) x^{n-1} + \textrm{low order terms}$. That is, evaluating the coefficient of $x^n$ in equation \eqref{Eq:TwoTermRec} yields $\gamma_n(m) = \lambda_{n+1}(m) - \lambda_{n}(m)$.

Orthogonality of $Q_j(x)$ demands that the coefficients $\vec{\lambda} = (\lambda_0(m), \ldots, \lambda_n(m))^T$ satisfy $\mathbf{H}_n(m) \vec{\lambda} = \vec{h}_n $ with $\vec{h}_n = -(1/(m+n)!, \ldots 1/(m+2n+1)!)^T$. Using Cramer's rule for $\lambda_n(m)$ and \eqref{Eq:JacIdentity}, one obtains the two term recursion for $\lambda_n(m)$
\begin{align*}
    \lambda_n(m) M_n(m) M_{n-2}(m+2) &= \lambda_{n-1}(m+2) M_{n-1}(m+2) M_{n-1}(m) \\
    &- \lambda_{n-1}(m+1) M_{n-1}(m+1)^2\;,
\end{align*}
which simplifies to
\begin{equation*}
    \lambda_n(m) = \tfrac{(m+2n-2)}{(n-1)} \lambda_{n-1}(m+1) - \tfrac{(m+n-1)}{(n-1)}\lambda_{n-1}(m+2)\;.
\end{equation*}
Induction then shows that 
\begin{equation*}
    \lambda_n(m) = -\tfrac{n}{m+2n - 1}, \; \forall n \geq 1, m \geq 0\;.
\end{equation*}
The orthogonal polynomials in this paper use the values $m = 0, 1$. In the case when $m = 1$, the coefficient $\lambda_n(1) = -\frac{1}{2}$ is constant, in which case $\gamma_n(1) = 0$, $\forall n \geq 1$.


\vspace{1.5em}
\bibliographystyle{plain}
\bibliography{references}

\vspace{1.5em}
\end{document}